\documentclass[11pt]{article}
\usepackage{amsmath, amssymb, amsthm}
\usepackage{geometry}
\usepackage{url, hyperref}
\usepackage[utf8]{inputenc}
\usepackage[french, english]{babel}
\usepackage{comment}

\newtheorem{THEOREM}{Theorem}[section]

\newtheorem{theorem}[THEOREM]{Theorem}
\newtheorem{lemma}[THEOREM]{Lemma}
\theoremstyle{definition}
\newtheorem{definition}[THEOREM]{Definition}
\newtheorem{claim}[THEOREM]{Claim}
\newtheorem{corollary}[THEOREM]{Corollary}
\newtheorem{example}[THEOREM]{Example}
\newtheorem{remark}[THEOREM]{Remark}
\newtheorem{Observation}[THEOREM]{Observation}
\newenvironment{observation}{\begin{Observation}}
{\end{Observation}}
\newtheorem{Notation}[THEOREM]{Notation}
\newenvironment{notation}{\begin{Notation}}
{\end{Notation}}

\newcommand{\forces}{\Vdash}
\newcommand{\dom}{\text{dom}}

\newcommand{\supt}{\text{supt}}

\newcommand{\ran}{\text{ran}}

\newcount\skewfactor
\def\mathunderaccent#1#2 {\let\theaccent#1\skewfactor#2
\mathpalette\putaccentunder}
\def\putaccentunder#1#2{\oalign{$#1#2$\crcr\hidewidth
\vbox to.2ex{\hbox{$#1\skew\skewfactor\theaccent{}$}\vss}\hidewidth}}
\def\name{\mathunderaccent\tilde-3 }

\newcommand{\rest}{\upharpoonright}
\newcommand{\deq}{\buildrel{\rm def}\over =}
\newcommand{\cf}{\text{cf}}

\newcommand{\eop}{\bigstar}
\newcommand{\DD}{\mathcal D}
\newcommand{\EE}{\mathcal E}
\newcommand{\FF}{\mathcal F}
\newcommand{\HH}{\mathcal H}
\newcommand{\V}{\mathbf V}

\title{A Baumgartner-style Property that Applies to Preservation of $\aleph_1$ and $\aleph_2$ under Iterations with Supports of Size $\aleph_1$}

\author{Mirna D{\v z}amonja, Logique Consult, Liège, Belgium}

\date{}

\begin{document}

\maketitle

\begin{abstract} We prove an iteration theorem along an arbitrary ordinal for forcings of which each adds a generic function on $\aleph_1$, while the iteration preserves $\aleph_1$ and $\aleph_2$. 
The theorem applies to forcings  that have closure and amalgamation properties similar to those present in the Perfect Set Forcing. The iteration is made with supports of size  $\aleph_1$. In the results we can replace $\aleph_1$ by any other cardinal $\kappa$ such that $\kappa^{<\kappa}=\kappa$.

The work is inspired by James E. Baumgartner's Axiom A forcings and his proof that iterations with countable support of the same preserve $\aleph_1$.  In honour of Baumgartner, the property introduced here is called Property B$(\kappa)$. The known additional difficulties when forcing at uncountable cardinals compared to the forcing at $\aleph_0$ make for less general theorems and a more complex theorem on the iteration, which is not an iteration theorem in the classical sense. The new element is that in addition to the properties of the individual iterands being strongly $(<\kappa)$-closed and exhibiting Property $B(\kappa)$ in a strong way, we also need to assume an amalgamation property of the whole iteration.

As a concrete example of a Property B$(\kappa)$ forcing which fits into the above iteration scheme, we consider a version of the Generalised Perfect Set Forcing, which we call Perfect with Respect to a Filter Forcing. We give its basic properties and show that the product of such forcing with supports $\le\kappa$ does not preserve $\kappa^+$ - hence an iteration is needed. We then successfully show that this forcing with respect to the filter of co-bounded subsets of  
$\kappa$, or filters forced by the iteration, fits into the iteration scheme discussed in the above. Together with developing the iteration scheme abstractly, applying it to these concrete examples is the main result of this paper.
\end{abstract}

\section{Introduction and Motivation} We start by briefly motivating the paper by the directions of research that it inscribes itself into. More detailed historical remarks are given in \S\ref{sec:history}. 

Starting with Baumgartner's lectures on Axiom A in Cambridge 1978 and subsequent development of proper forcing by Saharon Shelah from 1980 on, methods for iterating forcing with countable support while preserving $\aleph_1$ have become a main focus of research in set theory and have entered into the main stream of its applications. It is natural to try to extend these methods to cardinals larger than $\aleph_1$, notably $\aleph_2$. 
Several successes in this direction have been obtained by trying to generalise proper forcing and the corresponding forcing axiom PFA. The authors involved include Andrzej Ros{\l }anowski and Shelah with a series of papers started with \cite{zbMATH01751288}, Itay Neeman with a very influential idea using two kinds of elementary submodels as side conditions \cite{Neeman} and several other researchers developing Neeman's method, such as Giorgio Venturi and Boban Veli{\v c}kovi{\'c} \cite{VeVe}. 

However, it seems that the full problem of generalising proper forcing to $\aleph_2$ is difficult and at this time no analogue to PFA on $\aleph_2$ has been found. Moreover, there are combinatorial results by Shelah \cite{Shelah-flat}
and Tanmay Inamdar \cite{zbMATH07740598} that show that a generalisation of PFA to $\aleph_4$ cannot possibly hold, as provable just in ZFC. Inamdar and
Assaf Rinot \cite{Inamdar-Rinot-nonstructure-aleph2} have discovered many combinatorial properties of $\aleph_2$ that indicate that rather than a structure theory that $\aleph_1$ exhibits
under PFA, the cardinal $\aleph_2$ exhibits a non-structure theory just in ZFC. This makes it quite difficult to imagine that there should be an analogue of PFA
at $\aleph_2$.

Then it is interesting to ask, if we accept that having PFA at $\aleph_2$ is rather unlikely, can we have weaker iteration theorems that do apply to
$\aleph_2$. Versions of Generalised Martin's Axiom, although lacking in precision (see \cite{Mirna-Novosibirsk} for a corrected version) seem to have been known at least since 1977. See \S\ref{sec:history} for the historical details, but briefly, the idea goes back to the work of Richard Laver.  To our knowledge, there has not been an effort to generalise Axiom A forcing. That was the initial motivation of our paper. At a closer look, the difference between the proper forcing and Axiom A turns out to be practically non-existent (see \S\ref{sec:history}), and we may conclude that a general forcing axiom is not likely to exist. On the other hand, using a novel idea of iterations with amalgamation, we prove a meaningful iteration theorem with concrete applications. We briefly describe the genesis of these results. 

The difficulty of finding forcing arguments on $\aleph_2$ is to preserve both $\aleph_1$ and $\aleph_2$ in iterations with supports of size $\aleph_1$ and the fact that $\aleph_1$ tends to collapse in stages of countable cofinality.
We run into such difficulties even if we require the individual forcings to be countably closed and use countable support. 
In the two published versions of Generalised Martin Axiom to $\aleph_1$ obtained by countable support iteration
(by Baumgartner \cite[\S4]{Ba} and by Shelah in \cite{Sh80}), in addition to requiring individual forcing notions to satisfy countable closure, they also need to be well-met, and satisfy a stronger version of $\aleph_2$-cc ($\aleph_1$-linked for Baumgartner and stationary $\aleph_2$-cc for Shelah). We have shown in \cite{Mirna-Novosibirsk} that the condition of well-metness is neither sufficient nor necessary, but that one has to consider a definable version of $(<\kappa)$-closure which we call strong $(<\kappa)$-closure and which appears in this paper in Definition \ref{def:strongclosure}. Specifically, \cite{Mirna-Novosibirsk} shows that for $\kappa$ satisfying $\kappa^{<\kappa}=\kappa$, iterations of strongly $(<\kappa)$-closed stationary $\kappa^+$-cc forcing with supports of size $<\kappa$, preserve all cardinals up to and including $\kappa^+$. The question then becomes to what extent we can weaken the assumption on stationary $\kappa^+$-cc to include forcings which are not even $\kappa^+$-cc.

There are difficulties on the way. For example, in their generalisation of proper forcing, Ros{\l }anowski and Shelah \cite{zbMATH01751288} used countably closed forcing and a property they called properness over semi-diamonds, that is a specific use of $\diamondsuit$ in the ground model, to power the iteration. This generalisation does not seem to have found wide applications, probably since it is rather abstract and shares the problem of  \cite{Sh80} of the undefinability of the closure.
Almost any other forcing argument that we have about $\aleph_2$ goes through the procedure of collapsing a large cardinal to $\aleph_2$. This appeared first in Jack Silver's proof (see  \cite{Mitchellmodel}) that it is consistent modulo a weakly compact cardinal that there is a model with no $\aleph_2$-Aronszajn trees, and has been successfully used and developed in many contexts, including the Laver-Shelah \cite{LaverShelah} consistency proof of the 
$\aleph_2$-Suslin hypothesis. 

We are after an iteration method which is both simple to state and does not involve large cardinals. To this end,
we introduce property $B(\kappa)$ forcing and give a concrete example called Perfect with Respect to a Filter Forcing, denoted $\mathbb P(\FF)$. 
It is a generalisation of the Grigorieff Forcing from \cite{zbMATH03513777}. 

We study the behaviour of $B(\kappa)$-forcing under products and iterations. Unlike what happens for the Perfect Set Forcing, the side-to-side product
of $\mathbb P(\FF)$ with supports of size $\kappa$ fails to preserve $\kappa^+$, as is already the case with Grigorieff Forcing at $\kappa=\aleph_0$. The main result of the paper is a conditional iteration theorem for this type of forcing (Theorem \ref{th:iteration}) and the fact that it applies to ordinary iterations along an ordinal with supports of size $\le\kappa$ of the Perfect Set Forcing with Respect to the Filter of co-bounded subsets of $\kappa$, or to filters forced during the iteration (Corollary \ref{th:iterationPrikry} and Corollary \ref{cor:mixed-filters}). In the subsequent work \cite{Kanovei-style}, we have shown that in fact the iteration can be performed along any well-founded poset in place of an ordinal. In the recent article \cite{SP}, we have used a version of $\mathbb P(\FF)$ adapted to singular cardinals to solve the Solovay Problem of adding an unbounded subset to a singular cardinal without adding any bounded ones. In that work, our technology allowed us to obtain a ZFC result that works for any cofinality of the singular.

This paper is organised as follows. In \S\ref{recallA} we recall Axiom A and Baumgartner's iteration theorem for such forcing.
In \S\ref{sec:prepertyBdef} we define Property B$(\kappa)$ forcing and show that it preserves $\kappa^+$. The definition is such that Axiom A forcing is
exactly the property $B(\aleph_0)$.

Section \ref{gSacks}  is devoted to the Perfect with Respect to a Filter Forcing.

In \S\ref{subsec:productSacks}, we continue to products of Perfect with Respect to a Filter Forcing where we show that Perfect with Respect to a Filter Forcing does not behave well in products with support $\le\kappa$, in the sense that it collapses $\kappa^+$. This motivates the development of the iteration theory, which we do in \S\ref{subsec:iteration-general}. That section gives an abstract conditional iteration theorem for strongly $(<\kappa)$-closed property strongly $B(\kappa)$ forcing, with supports of size $\le\kappa$, but requiring a property of iteration called the amalgamation property, Theorem \ref{th:iteration}. In \S\ref{sec:whichones} we give concrete examples of its applicability, which include ${\mathbb P}(\FF)$ for $\FF$ the filter of co-bounded subsets of $\kappa$.

The paper ends with conclusions and further directions presented in \S\ref{sec:conclusions}, and historical remarks in \S\ref{sec:history}. In our opinion, these historical remarks are of importance, since it is exactly discovering them that has guided us in finding the correct formulations of what was needed for this work. 

Since no other forcing notion generalising Perfect Set Forcing to an uncountable cardinals is known to be iterable, we take the liberty to refer to the class of Perfect with Respect to a Filter Forcings as the Generalised Perfect Set Forcing.

\subsection{Notation and Traditions}
Set theorists commit a common fallacy by saying that objects are defined by (transfinite) induction, when in fact it is the recursion that defines and the induction that proves. We follow the tradition by stating that we define objects by induction.

Our forcing notation follows the so-called Jerusalem tradition, where $p\le q$ means that $q$ is a stronger condition than $p$. In addition to the tradition, this actually corresponds to our philosophical view of forcing as constructing a new object by a partial order of its approximations, shared by the historical origin of forcing in Muchnik's priority arguments in recursion theory. 

It is sometimes better to take the Boolean algebra view of narrowing down to a new object by an algebra of possible truth values, but it is not the case of this paper. Some comments on the two views of forcing and their semantical justification are given in \S\ref{sec:history}.

\section{Definition of Axiom A and our goals.}\label{recallA}

\begin{definition}\label{def:axiomA}[Axiom A]
A forcing notion $\mathbb P$ satisfies {\em Axiom A} if:
\begin{enumerate}
    \item There exist partial orders \(\leq_n\), with \(\leq_0 = \leq\), and $\le_{n+1}\subseteq \le_n$ for all $n$.
    \item If $\langle p_n:\,n<\omega\rangle$ is such that $p_n\le_n p_{n+1}$ for each $n$, there is $q$ such that for every $n$
    we have $p_n\le_n q$.
    \item\label{names} For any dense set $\mathcal D$ in $\mathbb P$, every $n<\omega$ and $p\in \mathbb P$, there is $q\ge_n p$ and a countable 
    $\mathcal D'\subseteq \mathcal D$ 
    such that $\mathcal D'$ is predense above $q$.
     \end{enumerate}
\end{definition}

\begin{remark}\label{dense-antichain} By usual translations between these concepts, see \cite[ChVII]{Kunen}, item (\ref{names}) in Definition \ref{def:axiomA} can be phrased equivalently in terms of antichains in place of dense sets or
in terms of names for the elements of the ground model. 
\end{remark}

We recall the first part of Theorem 7.1 from \cite{Ba}.

\begin{theorem}\label{th:preservingAxiomA} Suppose that $\langle {\mathbb P}_\alpha, \name{Q}_\beta :\,\alpha\le\alpha^\ast, \beta<\alpha^\ast\rangle$
is a countable support iteration of Axiom A forcings. 

Then $\aleph_1$ is preserved by the iteration.
\end{theorem}

We are interested to generalise this theorem to successor cardinals larger than $\aleph_1$, specifically to $\aleph_2$ \footnote{Theorem 7.1 from \cite{Ba}
has a second part which discusses the $\aleph_2$-cc property of such an iteration of length $\omega_2$ of size $\aleph_1$-forcings over a model of CH.
That part of the proof in \cite{Ba} is wrong and we do not discuss the analogue of it here.}.

\section{Property B$(\kappa)$}\label{sec:prepertyBdef}
In analogy with Baumgartner's Axiom A, we define Property B forcing, with respect to an inifnite cardinal $\kappa$.

\begin{definition}\label{def:propertyB}[{\em Property B$(\kappa)$}]
A forcing notion ${\mathbb P}=(P,\le)$ satisfies {\em Property B$(\kappa)$} if:
\begin{enumerate}
    \item There exist a sequence $\langle \leq_\zeta:\,\zeta<\kappa\rangle$ of partial orders on $P$, satisfying  \(\leq_0 = \leq\), and $\xi<\zeta\implies \le_\xi\supseteq \le_\zeta$, and such that:
    \item  For $\delta$ limit $\in (0,\kappa)$, we have $\le_\delta=\bigcap_{\zeta<\delta} \le_\zeta$.
  \item {\em (Fusion)} If $0<\delta\le\kappa$ is a limit ordinal we say that $\bar{p}=\langle p_\zeta:\,\zeta<\delta\rangle$ is 
  a $\delta$-fusion sequence if $p_\zeta\le_\zeta p_{\zeta+1}$ for each $\zeta$ and for $\varepsilon$ limit $<\delta$ we have that $p_\zeta\le_\zeta p_\varepsilon$
  for every $\zeta<\varepsilon$.
  
  The {\em Fusion} property states that for any $\delta$ limit $\le \kappa$ and any $\delta$-fusion sequence $\bar{p}$
  there is $q$ such that for every $\zeta<\delta$
 we have $p_\zeta\le_\zeta q$. We call such $q$ a {\em fusion limit} of $\bar{p}$.
   
  \item\label{mastering-def} {\em (Mastering $\kappa$)} For any dense set $\mathcal D$ in $\mathbb P$ and $p\in \mathbb P$, for every $\zeta<\kappa$, there is $q\ge_\zeta p$ and $\mathcal D'\subseteq \mathcal D$ of
    size $\le\kappa$
    such that $\mathcal D'$ is predense above $q$.
     \end{enumerate}
\end{definition}

\begin{Observation}\label{obs:masteq} (1) The definition implies that $B(\aleph_0)$ is exactly Axiom A forcing.
\smallskip

{\noindent (2)} The condition {\em Mastering $\kappa$} in Definition \ref{def:propertyB} can be rephrased as follows: if $p\forces``\name{\tau} \mbox{ is an ordinal}"$,
then there is $q\ge_\zeta p$ and a set of ordinals $X\in {\mathbf V}$ of size $\le\kappa$, such that $q\forces ``\name{\tau}\in X"$.

[To see that this property is implied by {\em Mastering $\kappa$} as stated in the definition, notice that 
\[
\DD=\{r:\,r \mbox{ forces a value to }\tau\}
\]
is dense. Let $q\ge_\zeta p$ and $\DD'\subseteq \DD$ of size $\le\kappa$ be such that $\DD'$ is predense above $q$. Let $X$ be the set
of ordinals $\alpha$ such that there is $r\in \DD'$ that forces $\name{\tau}$ to be $\alpha$. Hence $X\in {\mathbf V}$ is of size $\le\kappa$, and $q$ forces 
$\name{\tau}$ to be in
$X$ since every extension of $q$ is compatible with a condition in $\DD'$.

In the other direction, let $\DD=\{q_i:\,i<i^\ast\}$ and let $\name{\tau}$ be the $\mathbb P$-name for the first $i$ such that $q_i$ is in $G$, which is well
defined since $\DD$ is dense. In particular $\name{\tau}$ is a name for an ordinal. Let $p$ be any condition and $\zeta<\kappa$. Let $X$ and $q$ be as given by the property under observation. Define
$\DD'=\{q_i:\;i\in X\}$, so $\DD'\subseteq \DD$ is of size $\le\kappa$ and it is predense above $q$.]

\smallskip

{\noindent (3)} Item (3) of Definition \ref{def:propertyB} can be stated in an equivalent form using sequences of conditions of length the cofinality of $\delta$.
Namely let $\langle \delta_i:\,i<\cf(\delta)\rangle$ be cofinal in $\delta$. Then a sequence $\langle p_\zeta:\,\zeta<\delta\rangle$ like the one in the requirements of item (3) will
have an upper bound in $\le_\delta$ iff the sequence $\langle p_{\delta_i}:\,i<\cf(\delta)\rangle$ has such an upper bound, given the assumptions on the orders $\le_\zeta$.

\end{Observation}

\begin{lemma}\label{preservation-kappa-plus} Property B$(\kappa)$ forcing preserves $\kappa^+$.
\end{lemma}

\begin{proof} It suffices to show that in $\mathbf V[G]$ no function from $\kappa$ to $\kappa^+$ can be cofinal. Let $p$ force $\name{f}$
to be a function from $\kappa$ to $\kappa^+$ and suppose for a contradiction that $p$ forces that the range of
$\name{f}$ is cofinal in $\kappa^+$. 

We are going to build a condition $q\ge p$ as the limit of a fusion sequence $\bar{p}=\langle p_\zeta:\,\zeta<\kappa\rangle$ and a function $g\in \mathbf V$
as follows. This sequence and the value of $g(\zeta)$ will be found by induction on $\zeta$ so that:
\begin{itemize}
\item $p_0=p$,
\item $p_{\zeta+1}\ge_\zeta p_\zeta$,
\item $g(\zeta)$ is a set in $\mathbf V$ of size $\le\kappa$ and $p_{\zeta+1}\forces``\name{f}(\zeta)\in g(\zeta)$".
\end{itemize}
Arriving to stage $\zeta+1$, we apply {\em Mastering $\kappa$} with $\name{\tau}$ being $\name{f}(\zeta)$ and $p$ being $p_\zeta$. We let 
$g(\zeta)=X$ for a set $X$ provided by Observation \ref{obs:masteq}(2)  and let $p_{\zeta+1}$ be the $r$ provided by the lemma. We can continue at non-zero limit
stages $\delta$ by {\em Fusion} applied to $\delta$. {\em Fusion} also lets us find $q$ at the end of the induction,
as the fusion limit of the sequence $\bar{p}$.

Notice that the function $g$ is in $\V$. For any $\zeta<\kappa$ we have that $|g(\zeta)\cap \kappa^+|\le \kappa$. Hence there is $\beta\in \kappa^+ \setminus \sup( \ran(g))$.
So $q\forces ``\beta>\sup(\ran(\name{f}))"$, in contradiction to $p$ forcing $\name{f}$ to be cofinal in $\kappa^+$.
$\eop_{\ref{preservation-kappa-plus}}$
\end{proof}

\section{Perfect with Respect to a Filter Forcing, or the Generalised Perfect Set Forcing}\label{gSacks} 
We introduce a forcing which is a generalisation of the $\kappa$-Perfect Set Forcing of Victoria Gitman and Sy-David Friedman 
from \cite[\S2]{FriedmanGitman:ModelOfACNotDCInaccessible}\footnote{Various incorrect versions of both $\kappa$-Perfect Set Forcing and its generalisation with filters have appeared in the literature. See \S\ref{sec:history} for historical details.}.

Let $\kappa$ be an infinite regular cardinal and let $\FF$ be a $(<\kappa)$-complete filter on $\kappa$. We say that a subtree $p\subseteq {}^{{<\kappa}}\kappa$ is  {\em $\FF$-perfect} if 
\begin{itemize}
\item $\sup\{\dom(s):\,s\in p\}=\kappa$, 
\item for every $s\in p$, either $| \{\alpha<\kappa:\,s\frown \alpha\in p\} | = 1$, or $\{\alpha<\kappa:\,s\frown \alpha\in p\}\in \FF$ and
\item for any $s\in p$, there is
$t\supseteq s$ such that $\{\alpha<\kappa:\, t\frown\alpha\in p\} \in \FF$.
\end{itemize}
We say that $t$ such as in the last property is $\FF$-{\em splitting}. 
For $s\in p$ we define its {\em splitting history} as
\[
{\rm deg}_p (s)=\{i < \dom(s):\,(\exists t\in p)\,[t\rest i =s\rest i \mbox{ and }t(i)\neq s(i)]\}.
\]

\begin{definition}\label{gSacks:def} (1) The $\FF$-Perfect Tree forcing $\mathbb P=\mathbb P(\FF)$ is defined as follows:

The conditions in $\mathbb P$ are $\FF$-perfect subtrees $p$ of ${}^{{<\kappa}}\kappa$, satisfying the following additional conditions:
\begin{description}

\item [{\rm (a) [Closure of $p$]}] If $\delta<\kappa$ is a limit ordinal and $s\in {}^{\delta} \kappa$ is such that for all $\alpha<\delta$ we have $s\rest\alpha\in p$, then $s\in p$.
\item [{\rm (b) [Closure of splitting for every node]}] For every $s\in p$, the set ${\rm deg}_p (s)$ is a closed subset of $\dom(s)$.
\end{description}
The order is defined by $p\le q$ if $p\supseteq q$.

The {\em stem} of $p$,
denoted by $s[p]$ is defined by
\[
s[p]=\bigcup \{s\in p:\,\mbox{no }t\subset s \mbox{ is splitting and }s \mbox{ is splitting}\}.
\]
Note that this is well-defined, since every node in an $\FF$-perfect tree has a splitting extension, while such a tree is rooted. So, there is a unique least splitting extension of the root, which is the stem.

\smallskip

{\noindent (2)} For $\zeta<\kappa$ the order $\le_\zeta$ is defined by letting 
$p\le_\zeta q$ if 
\begin{itemize}
\item $p\le q$,
\item for all $s\in p$ with ${\rm otp}({\rm deg}_p(s))< \zeta$ we have $s\in q$.
\end{itemize}

\smallskip

{\noindent (3)} For $p$ as above, we let $[p]$ be the set of $\kappa$-branches of $p$, namely
\[
[p]=\{f\in {}^\kappa\kappa:\,(\forall \alpha<\kappa) f\rest\alpha\in p\}.
\]

\end{definition} 

Keeping the notation from Definition \ref{gSacks:def}, we make several observations. 

\begin{observation}\label{obs:basicSacks} {\noindent (0)} If $p\supseteq q$ and $s\in q$, then ${\rm deg}_p (s)\supseteq {\rm deg}_q (s)$.
Each $\le_\zeta$ is a transitive relation.

[The first statement follows because $p\supseteq q$. If $p\le_\zeta q$ and $q\le_\zeta r$ then $p\supseteq q \supseteq r$ and hence $p\supseteq r$.
If $s\in p$ is such that ${\rm otp}({\rm deg}_p(s))< \zeta$, then $s\in q$ by the definition of $\le_\zeta$ and ${\rm otp}({\rm deg}_q(s))
\le  {\rm otp}({\rm deg}_p(s)) < \zeta$, hence $s\in r$.]

{\noindent (1)} $p\le_{\zeta+1} q$ implies that $q\cap {}^{\zeta} \kappa=p\cap {}^{\zeta} \kappa$.

[Since $q\subseteq p$, we clearly have $q\cap {}^{\zeta} \kappa\subseteq p\cap {}^{\zeta} \kappa$. For the other direction, suppose that $s\in p\cap {}^{\zeta} \kappa$. Then ${\rm deg}(s)\subseteq \dom(s)=\zeta$, hence ${\rm otp}({\rm deg}(s))\le \zeta<\zeta+1$.
So, if $p\le_{\zeta+1} q$, we have $s\in q$.]

{\noindent (2)} If $\zeta\le\xi$ and $p\le_\xi q$, then $p\le_\zeta q$. If $0<\delta<\kappa$ is a limit, then $\le_\delta=\bigcap_{\zeta<\delta}\le_\zeta$, and in fact
$\le_\delta=\bigcap_{\xi\le\zeta<\delta}\le_\zeta$ for any $\xi<\delta$.

[Follows by the definitions.]

{\noindent (3)} For $\delta$ limit $\in (0,\kappa)$, if $\langle p_\zeta:\,\zeta<\delta\rangle$ and $q$ are such that $p_\zeta\le_\zeta q$ for every 
 $\zeta<\delta$, then  $p_\zeta\le_\delta q$ for every $\zeta<\delta$.
 
 [Follows by (2).]
 
{\noindent (4)} If $f\in [p]$, then $C[f]=\{\alpha<\kappa:\,\{i<\kappa:\,f\rest \alpha\frown i\in p\}\in \FF\}$ is a club of $\kappa$.

[To see that $C[f]$ is unbounded, let $\alpha_0<\kappa$ be arbitrary and let $s=f\rest\alpha_0$, hence $s\in p$. Let $s\subseteq t$ for some $t$ such that
$t$ is $\FF$-splitting in $p$, where $t$ is the shortest such node. Hence $t\subseteq f$ and letting $\alpha=\dom(t)$ we obtain $\alpha_0\le\alpha$ and
$\alpha\in C[f]$.

To see that $C[f]$ is closed, suppose that $\delta<\kappa$ is such that $\delta=\sup(C[f]\cap\delta)$ and let $s=f\rest(\delta+1)$. Then we have that 
$C[f]\cap\delta={\rm deg}_p (s)\cap\delta$, so by the closure of splitting of every node, we have that $\delta\in {\rm deg}_p(s)$. Hence $\delta\in C[f]$ as exemplified by $s\rest\delta$.]
\end{observation}

\begin{notation}\label{s-tree} Suppose that $p\in \mathbb P$ and $s\in p$. Let
\[
p_s=\{t\in p:\,t\subseteq s\vee s\subseteq t\}.
\]
Note that under these circumstances $p_s\in  \mathbb P$ and $p_s\ge p$.
\end{notation}

\begin{lemma}\label{lem:generic-kappa-function} Suppose that $G$ is a $\mathbb P(\FF)$-generic filter. Then 
\[
g=\bigcup \{s[p]:\,p\in G\}
\]
is a function from $\kappa$ to $\kappa$ in $\mathbf V[G]$.
\end{lemma}

\begin{proof} To show that $g$ is a function, suppose that $p,q\in G$. Hence there is $r\in G$ with $p,q\le r$ and in particular $s[r]\supseteq s[p], s[q]$.
Hence $s[p]$ and $s[q]$ are compatible as functions.

It follows from this and the definition of $\mathbb P(\FF)$ that $g$ is a function whose domain is some $\alpha\le\kappa$. To show that the domain is
actually $\kappa$ we need to demonstrate that for every $\alpha<\kappa$ the set $\DD_\alpha=\{q\in \mathbb P(\FF):\,\dom(s[q])\ge\alpha\}$ is dense.
Suppose that $p\in \mathbb P(\FF)$. By the definition of what an $\FF$-perfect tree is, $p$ has a node $s$ of height $\ge\alpha$. Then by
letting $q=p_s$ we obtain that $p\le q$ and $s[q]= s$. Therefore $q\in \DD_\alpha$.
$\eop_{\ref{lem:generic-kappa-function}}$
\end{proof}

\begin{notation}\label{genPrikrygeneric} The $g$ from $\kappa$ to $\kappa$ as in Lemma \ref{lem:generic-kappa-function} is called {\em the generic function added by} $\mathbb P(\FF)$.
\end{notation}

\begin{lemma}\label{when-closed} (1) ${\mathbb P}(\FF)$ is $(<\kappa)$-closed and in fact every increasing sequence 
$\bar{p}=\langle p_\zeta:\,\zeta<\lambda\rangle$ for any $0<\lambda<\kappa$ has the least upper bound, namely $q=\bigcap_{\zeta<\lambda} p_\zeta$.

{\noindent (2)} Suppose that $\varepsilon <\kappa$ and $\bar{p}=\langle p_\zeta:\,\zeta<\lambda\rangle$ is in fact $\le_\varepsilon$-increasing. Then 
$q$ is also the least upper bound of $\bar{p}$ in the order $\le_\varepsilon$.
\end{lemma}

\begin{proof}  (1) The lemma is trivially true for $\kappa=\aleph_0$, so let us assume that $\kappa>\aleph_0$. Given $\lambda <\kappa$ and an increasing sequence $\bar{p}=\langle p_\zeta:\,\zeta<\lambda\rangle$  in $\mathbb P$. 

Let $q=\bigcap_{\zeta<\lambda} p_\zeta$. Once we check that $q$ is a condition and an upper bound of $\bar{p}$, it will immediately follow that it is the lub. 
If $\lambda$ is a successor ordinal $\iota+1$, then the lemma is trivially true since then $q=p_\iota$. Hence, let us assume that $\lambda$ is a limit ordinal.
Clearly $\langle \rangle \in q$, so $q\neq \emptyset$ and clearly $q$ is downward closed since each $p_\zeta$ is.

\begin{claim}\label{cl:splitting extension} Every node in $q$ has an $\FF$-splitting extension.
\end{claim}

\begin{proof}[{\bf Proof of the Claim}] Let $s\in q$.
For $\zeta<\lambda$ let $t_\zeta\supseteq s$ be an $\FF$-splitting node in $p_\zeta$ with the least possible height and let 
$A_\zeta=\{i <\kappa:\,t_\zeta \frown i \in p_\zeta\}$. Hence the sequence $\langle \dom(t_\zeta):\;\zeta<\lambda\rangle$ is non-decreasing.

Suppose first that this sequence stabilises, so there is $\zeta^\ast<\lambda$ such that for all $\zeta\in [\zeta^\ast, \lambda)$ we have that $t_\zeta$ is a fixed node
$t$. Then $t\in q$ and $A\deq \bigcap_{\zeta\in [\zeta^\ast, \lambda)}  A_\zeta \in \FF$, by the $(<\kappa)$-completeness of $\FF$. Moreover, if $i\in A$, then $t\frown i\in p_\zeta$ for every 
$\zeta\in [\zeta^\ast, \lambda)$ and hence $t$ is an $\FF$-splitting node in $q$.

Otherwise, the sequence $\langle \dom(t_\zeta):\;\zeta<\lambda\rangle$ has a strictly increasing cofinal subsequence of length $\cf(\lambda)$. Let us denote that subsequence by $\langle o_\iota:\,\iota<\cf(\lambda)\rangle$ and let $r_\iota=p_\zeta$ for the least $\zeta$ such that $\dom(t_\zeta)=\dom(o_\iota)$. Note that in this case we
have $q=\bigcap_{\iota<\cf(\lambda)} r_\iota$ since $\langle r_\iota:\,\iota<\cf(\lambda)\rangle$ is $\le_{\mathbb P(\FF)}$-cofinal in $\langle p_\zeta:\,\zeta<\lambda\rangle$. Let
$t=\bigcup_{\zeta<\lambda} t_\zeta=\bigcup_{\iota<\cf(\lambda)} o_\iota$. Then by the closure property of each $r_\iota$ we have $t\in r_\iota$ for each 
$\iota$ and hence $t\in q$. Moreover,
by the closure of splitting for every node in every $r_\iota$, we have that $t$ is $\FF$-splitting in every $r_\iota$, and hence in every $p_\zeta$. Let $S_\zeta=\{i <\kappa:\,t \frown i\in p_\zeta\}$, so
the sequence $\langle S_\zeta:\,\zeta<\lambda\rangle$ is a decreasing sequence of elements of $\FF$. By the $(<\kappa)$-completeness of $\FF$ we have
that $S\deq \bigcap_{\zeta<\lambda} S_\zeta\in \FF$ and by the definition of $q$ we have $t\frown i\in q$ for every $i\in S$. Hence $t$ is an $\FF$-splitting 
extension of $s$ in $q$.
$\eop_{\ref{cl:splitting extension}}$
\end{proof}

\begin{claim}\label{cl:height} The height of $q$ is $\kappa$.
\end{claim}

\begin{proof}[{\bf Proof of the Claim}] First note that the situation of $\lambda$ being a successor ordinal $\zeta+1$ is clear, since then $q=p_{\zeta}$. So, without
loss of generality $\lambda$ is a limit ordinal $>0$.

Suppose for a contradiction that the claim is false, so let $\alpha<\kappa$ be the first ordinal $<\kappa$ such that $q$ has no nodes of level $\alpha$. We have $\alpha>0$ since $\langle\rangle \in q$.

If $\alpha$ is a limit ordinal, let $\langle s_i:\,i<\alpha\rangle$ be an increasing sequence of nodes of $q$ with $\sup({\rm ht}(s_i))=\alpha$. The existence of
such a sequence follows from the fact that every node in $q$ has an $\FF$-splitting extension, as we have shown in Claim \ref{cl:splitting extension}.

Hence $s=\bigcup_{i<\kappa} s_i$ is a function with domain $\alpha$. By the requirement (a) on the closure of conditions in ${\mathbb P}(\FF)$, we have that for every
$\zeta<\lambda$ the function $s$ is an element of $p_\zeta$. Hence $s\in q$, a contradiction.

So we conclude that $\alpha=\beta+1$ for some ordinal $\beta<\kappa$. Let $s\in q$ be on the level $\beta$, so $s$ is maximal. Let $t$ be an $\FF$-splitting extension of $s$ which is in $q$, and whose existence we have shown in Claim \ref{cl:splitting extension}. By the choice of $\alpha$, it follows that $t=s$, so $s$ is $\FF$-splitting. This in particular means that $s\frown i \in q$ for some $i$, and this is a contradiction with $s$ being a maximal element.
$\eop_{\ref{cl:height}}$
\end{proof}

\begin{claim}\label{cl:Fsplitting}  If $s\in q$ is a splitting node in $q$, then it is an $\FF$-splitting node.
\end{claim}

\begin{proof}
If $s\in q$ is a splitting node in $q$, then it is a splitting node in each $p_\zeta$, with the set $A_\zeta=\{i <\kappa:\,t_\zeta \frown i \in p_\zeta\}$ in
$\FF$. Hence $A\deq \bigcap_{\zeta<\lambda} A_\zeta\in \FF$ and for each $i\in A$ we have $s\frown i\in q$. Therefore $s$ is $\FF$-splitting in $q$.
$\eop_{\ref{cl:Fsplitting}}$
\end{proof}

\begin{claim}\label{closure-Fsplitting}
Suppose that $\delta<\kappa$ is a limit ordinal and that for $s\in {}^\delta\kappa$ we have
\[
\sup\{\alpha<\delta: s\rest \alpha\in q  \mbox{ [respectively, $s\rest\alpha$ splits in }q]\}=\delta.
\]
Then $s\in q$ [respectively, $s$ is a splitting node in $q$].
\end{claim}

\begin{proof} Under the given assumptions, we have
\[
\sup\{\alpha<\delta: s\rest \alpha\in p_\zeta \mbox{[respectively, $s\rest\alpha$ splits in }p_\zeta]\}=\delta,
\]
for every $\zeta<\lambda$. Hence $s\in p_\zeta$ [respectively $s$ splits in $p_\zeta$] for every $\zeta<\lambda$. By the definition of $q$ and using the
$(<\kappa)$-completeness of $\FF$, we conclude that $s\in q$ [respectively, $s$ splits in $q$], 
$\eop_{\ref{closure-Fsplitting}}$
\end{proof}

The claims prove part (1) of the lemma.

\smallskip 
{\noindent (2)} Since $q$ is the least upper bound of $\bar{p}$ in $\le$ and $\le_\varepsilon\subseteq \le$, if we prove that $q$ is a $\le_\varepsilon$
upper bound of $\bar{p}$, it will follow immediately that it is the least upper bound.

If $\lambda$ is a successor ordinal, the lemma follows since $q$ is the last element of 
$\bar{p}$ and $\le_\varepsilon$ is transitive.

Suppose that $\lambda$ is a limit ordinal and let $\zeta<\lambda$, we need to prove that $p_\zeta\le_{\varepsilon} q$. Let $s\in p_\zeta$
be such that ${\rm otp}(\deg_{p_\zeta}(s))<\varepsilon$. The assumption that $\bar{p}$ is $\le_{\varepsilon}$ increasing means that
$s\in p_\xi$ for any $\xi\in [\zeta, \lambda)$. Hence $s\in q$ and so $p_\xi\le_{\varepsilon} q$.
$\eop_{\ref{when-closed}}$
\end{proof}

\begin{lemma}[{\bf Fusion Lemma}]\label{Sacks:Fusion} Suppose that $0<\delta\le \kappa$ is a limit ordinal and $\bar{p}=\langle p_\zeta:\,\zeta<\delta\rangle$ is such that $p_\zeta\le_{\zeta}  p_{\zeta+1}$ for each $\zeta$, while for limit ordinals $0<\varepsilon <\delta$ we have $p_\varepsilon=\bigcap_{\zeta<\varepsilon} p_\zeta$.

Then $q=\bigcap_{\zeta<\delta} p_\zeta$ is a condition such that for every $\zeta<\delta$ we have $p_\zeta\le_\zeta q$. Moreover, it is the least condition with that property.
\end{lemma}

\begin{proof} Let us first consider the situation $0<\delta<\kappa$. We shall prove that for any $\zeta\le\delta$ we have $p_\zeta\le_\delta p_\delta$.

Lemma \ref{when-closed}(2) implies that $q$ is a condition which is the least upper bound of $\bar{p}$ in any of the orders $\le_\varepsilon$ for $\varepsilon<\delta$. 
Hence for any fixed ordinal  $\zeta<\delta$, it follows that $p_\zeta\le_\varepsilon q$ for any $\varepsilon <\delta$. Since $\le_\delta=\bigcap_{\varepsilon <\delta}\le_\varepsilon$, we have  $p_\zeta\le_\delta q$. Hence $q$ is a common $\le_\delta$-upper bound of $\bar{p}$ and since by Lemma
\ref{when-closed}(1) it is the least upper bound in $\le$, then it is also the least upper bound in $\le_\delta$.

Now let us assume $\delta=\kappa$.

Clearly $q$ is subtree of ${}^{{<\kappa}}\kappa$ and $q\subseteq p_\zeta$ for any $\zeta<\kappa$. It is also clear that if $q$ satisfies the properties required, it is the least condition with these properties, in the sense that if $r$ has the same properties then $q\le r$. So let us concentrate on showing the required properties of $q$.

Suppose that $\zeta<\kappa$ and $s\in p_\zeta$ satisfies ${\rm otp}({\rm deg}_{p_\zeta}(s))< \zeta$ while $\zeta\le\xi<\kappa$. By the transitivity of
the relation $\le\zeta$ and the fact that $\zeta<\xi\implies \le_\zeta\supseteq \le_{\xi}$, we have that $p_\zeta\le_\zeta p_\xi$. Therefore $s\in p_\xi$.
Since $\xi$ was arbitrary, we obtain $s\in q$. This shows $p_\zeta\le_\zeta q$.

This argument can be used to show the closure of splitting in $q$. Namely, let $s\in q$ and consider $\deg_q(s)$. Let $\zeta<\kappa$ be such that ${\rm otp}(\deg_q(s))<\zeta$. Then for every $t\subseteq s$ we have $t\in p_\zeta$ and ${\rm otp}(\deg_{p_\zeta}(t))<\zeta$. In particular, $\deg_q(s)$ is the
same set as $\deg_{p_\zeta}(s)$, which we know is a closed subset of $\dom(s)$.

The above argument also shows that the height of $q$ is $\kappa$ since we can construct by induction on $\zeta$ an increasing sequence of functions 
$\langle s_\zeta:\,\zeta<\kappa\rangle$ such that $s_\zeta\in p_{\zeta+1}$ and ${\rm otp}({\rm deg}_{p_{\zeta+1}}(s_\zeta))=\zeta$. The existence of
such $s_\zeta$ follows by the fact that $p_{\zeta+1}$ is perfect.
By the above, each $s_\zeta\in q$ and since ${\dom}(s_\zeta)\ge \zeta$, we have that the height of $q$ is at least $\kappa$. Of course, it 
cannot be greater than $\kappa$, so it is $\kappa$.

The closure of $q$ follows by definition, since if $s\in {}^{\delta} \kappa$ and $0<\delta<\kappa$ is a limit ordinal such that 
$\sup\{\alpha<\delta:\,s\rest\alpha\in q\}=\delta$, then for every $\zeta<\kappa$ we have $\sup\{\alpha<\delta:\,s\rest\alpha\in p_\zeta\}=\delta$ and
hence $s\in p_\zeta$ for all $\zeta$. Therefore $s\in q$. 

To finish, we need to prove the two remaining requirements for $q$ to be $\FF$-perfect, so let $s\in q$. 

First let us show that $s$ has a splitting extension in $q$.
Let $\alpha={\rm otp(deg}_q(s))$. Note that $\zeta\le\xi<\kappa\implies p_\zeta\supseteq p_\xi$ and hence ${\rm deg}_{p_\zeta}(s) \supseteq {\rm deg}_{p_\xi}(s)$, so
${\rm otp(deg}_{p_\zeta}(s))\ge {\rm otp(deg}_{p_\xi}(s))$.
Since then
$\langle {\rm otp(deg}_{p_\zeta}(s)):\,\zeta<\kappa \rangle$ is a decreasing sequence of ordinals, it must stabilise, so there is $\zeta^\ast$ such that for all $\zeta\in [\zeta^\ast,\kappa)$ we have
${\rm otp(deg}_q(s))=\alpha$. Without loss of generality, $\zeta^\ast>\alpha+1$. Let $t\supseteq s$, $t\in p_{\zeta^\ast}$ and $t$ is an $\FF$-splitting node
in $p_{\zeta^\ast}$ with no $\FF$-splitting nodes between $s$ and $t$. Hence ${\rm otp(deg}_{p_{\zeta^\ast}}(t))=\alpha$ and for any $i$ such that $t\frown i
\in p_{\zeta^\ast}$, we have
 ${\rm otp(deg}_{p_{\zeta^\ast}}(t\frown i))=\alpha+1< \zeta^\ast$. Therefore for all such $i$ we have that $t\frown i \in p_\zeta$ for any $\zeta\ge \zeta^\ast$
 and, in particular, the set $\{i<\kappa:\, t\frown i \in p_\zeta\}$ is a fixed set $A\in \FF$, for all $\zeta\in [\zeta^\ast, \kappa)$.
 Hence $t\in q$ is an $\FF$-splitting node, as witnessed by $A$, and is above $s$.
 
If $s$ is splitting in $q$ then $t=s$ in the above, so we can conclude that $s$ is $\FF$-splitting. 

These observations finish the proof of the lemma.
$\eop_{\ref{Sacks:Fusion}}$
\end{proof}

Now we prove Lemma \ref{Miller-density} which shows that under the assumption $\kappa^{<\kappa}=\kappa$, we have the {\em Mastering $\kappa$}
property.
\begin{lemma}[{\bf Mastering $\kappa$ Lemma}]\label{Miller-density} Assume $\kappa^{<\kappa}=\kappa$. Then 
$\mathbb P(\mathcal F)$ satisfies {\em Mastering $\kappa$}.
\end{lemma}

\begin{proof} Suppose that $\zeta<\kappa$ and $p\forces``\name{\tau}\mbox{ is an ordinal}"$.
Let $\{s_\alpha:\,\alpha\le \alpha^\ast\}$ be an enumeration of all $s\in p$ with ${\rm otp}(\deg_p (s))=\zeta$ and such that $s$ is a splitting node in $p$, so
$|\alpha^\ast|\le\kappa=|p|$.
For each $\alpha$ and $i<\kappa$ such that $s_\alpha\frown i\in p$, let $r_{\alpha, i}\ge p_{s_\alpha\frown i}$ be some condition that decides the value of $\name{\tau}$, say $\gamma_{\alpha, i}$. Let $X=\{\gamma_{\alpha, i}:\,\alpha<\alpha^\ast, i<\kappa\}$, so $|X|\le\kappa$.

Now let $r=\bigcup_{\alpha<\alpha^\ast, i<\kappa} r_{\alpha, i}$. We claim that $r\in \mathbb P$ and $r\ge_{\zeta+1} p$. For the former claim, let us note that the conditions $r_{\alpha, i}$ are pairwise disjoint above their stems, since they have different stems. 
Each is $\FF$-perfect, so, their union $r$ is also 
$\FF$-perfect. For the closure property, note that if $\delta>0$ is a limit ordinal, $t\in {}^{\delta}\kappa$ and $t\rest \beta \in r$ for all $\beta<\delta$, then $t$ is either contained in an $s_\alpha$, so $t\in r$, or $t$ strictly contains $s_\alpha\frown i$ for some $\alpha$ and $i$, and is hence an element of $r_{\alpha,i}$. The closure of splitting is proved in a similar way.

For the latter claim, first note that 
the elements $t$ of $p$ satisfying ${\rm otp}(\deg_p (t))=\zeta+1$ and having the smallest height with that property are exactly of the form $s_\alpha\frown i$ for some $\alpha<\alpha^\ast$ and $i$ such that $s\frown i\in p$. That is, since $s$ is splitting and ${\rm otp}(\deg_p (s))=\zeta$, we have ${\rm otp}(\deg_p (s\frown i))=\zeta+1$. If $t$ has the stated property, by the minimality of its height, it must be of the form $s\frown i$ for some $s$ splitting with ${\rm otp}(\deg_p (s))=\zeta$, so $s=s_\alpha$ for some $\alpha$. By the choice of $r_{\alpha, i}$s, we have that each relevant $t$ is in some $r_{\alpha, i}$, so in $r$. All the other $o$ with ${\rm otp}(\deg_p (o))=\zeta+1$ are such that for some $t$ of the above form we have $t\subseteq o$ and there are no splittings in 
$p$ between $t$ and $o$. Hence if $t\in r_{\alpha, i}$, then $o\in r_{\alpha, i}$ as well and so $o\in r$. 

Now suppose that $q\ge r$
and $q$ decides the value of $\name{\tau}$. Then $q$ must be compatible with some $r_{\alpha, i}$ and hence the value decided by $q$ for $\name{\tau}$
must be one among the $\gamma_{\alpha, i}$'s. Therefore $r\forces``\name{\tau}\in X"$.
$\eop_{\ref{Miller-density}}$
\end{proof}

\begin{corollary}\label{Miller-propertyB} If $\kappa^{<\kappa}=\kappa$ and $\mathcal F$ is $(<\kappa)$-complete, then the $\FF$-perfect tree forcing satisfies
Property B$(\kappa)$.
\end{corollary}

\begin{proof}  We obtain {\em Fusion} by Lemma \ref{Sacks:Fusion} and {\em Mastering $\kappa$} by Lemma \ref{Miller-density}.
$\eop_{\ref{Miller-propertyB}}$
\end{proof}

\begin{notation}\label{not:amalgam} For future reference, we shall name the construction used in the proof of Lemma \ref{Miller-density}. Given $p, \zeta$ and $\name{\tau}$ as in the
beginning of that proof, we shall refer to 
$\{s_\alpha:\,\alpha\le \alpha^\ast\}$ as the $\zeta$-front of $p$ and to any $r$ as constructed there as {\em an ($\zeta+1, \name{\tau}$)-amalgam of $p$}.
\end{notation}

To conclude, we summarise the cardinal preservation properties of $\mathbb P(\FF)$.

\begin{lemma}\label{Sacks:preservation} Forcing with $\mathbb P(\FF)$ preserves the following cardinalities and cofinalities:
\begin{enumerate}
\item $\le\kappa$,
\item $\kappa^+$ if $\kappa^{<\kappa}=\kappa$,
\item $\ge\kappa^{++}$ if $\kappa^{<\kappa}=\kappa$ and $2^\kappa=\kappa^+$.
\end{enumerate}
\end{lemma}

\begin{proof} For cardinalities and cofinalities $\le\kappa$, the conclusion follows by the $(<\kappa)$-closure of ${\mathbb P}({\mathcal F})$.

Suppose that $\kappa^{<\kappa}=\kappa$. By Corollary \ref{Miller-propertyB}, we have that ${\mathbb P}(\FF)$ satisfies Property B$(\kappa)$, 
and hence Lemma \ref{preservation-kappa-plus} shows that $\kappa^+$ is preserved.

Suppose that $\kappa^{<\kappa}=\kappa$. Then every condition in ${\mathbb P}({\mathcal F})$ is a subset of a fixed set ${}^{<\kappa}\kappa$
whose size is $\kappa$. Therefore the size of ${\mathbb P}({\mathcal F})$ is $\le 2^\kappa$. If we assume $2^\kappa=\kappa^+$,
it follows that the forcing has $\kappa^{++}$-cc and so it preserves cardinals $\ge\kappa^{++}$.
$\eop_{\ref{Sacks:preservation}}$
\end{proof}

If $\kappa=\aleph_0$, then every forcing is $(<\kappa)$-closed. Brown and Groszek \cite[\S1]{zbMATH05121369} mention that the requirement that $\FF$ is $(<\kappa)$-complete is necessary for the 
$(<\kappa)$-closure of $\mathbb P(\FF)$ for $\kappa>\aleph_0$. The following example shows why.

\begin{example}\label{not-always-closed} If $\kappa>\aleph_0$, and the forcing $\mathbb P(\FF)$ is $(<\kappa)$-closed, then $\FF$ must be $(<\kappa)$-complete.
\end{example}

\begin{proof} Suppose that $\kappa>\aleph_0$, $\lambda<\kappa$ and there is a sequence $\langle A_i:\;i<\lambda\rangle$ of sets in $\FF$ such that $\bigcap_{i<\lambda} A_i\notin \FF$. For $i<\lambda$ let $p_i$ be the unique condition in $\mathbb P$ in which every node has splittings in 
$\bigcup_{j\ge i} A_j$. Then $\langle p_i:\,i<\lambda\rangle$ is an increasing sequence in $\mathbb P(\FF)$ but $\bigcap_{i<\lambda} p_i\notin 
{\mathbb P}(\FF)$. On the other hand, if $q \in {\mathbb P}(\FF)$ is any condition such that $p_i\le q$ for all $i<\lambda$ then in particular
$q\subseteq \bigcap_{i<\lambda} p_i$, a contradiction.
$\eop_{\ref{not-always-closed}}$
\end{proof}

\section{A Product of the Perfect With Respect to a Filter Forcing}\label{subsec:productSacks} This product looks like the side-to-side product of Perfect Set Forcing, which is
described in Baumgartner's article \cite[\S1]{zbMATH03926906}. However, as opposed to the product of Perfect Set Forcing, the product described here does not preserve $\kappa^+$. In fact, already the product of Grigorieff forcing does not preserve $\aleph_1$, as was known and will be recalled  below.

We start by defining the forcing and by showing easy facts that do work out for this example.  Let $\kappa$ be an infinite regular cardinal, $\FF$ a $(<\kappa)$-complete filter on $\kappa$ and $\lambda\ge 1$ a cardinal.
We refer to \S\ref{gSacks} for the definitions relating to $\mathbb P=
\mathbb P(\FF)$. 
\begin{definition}\label{gSacks:def-lambda} (1) The $\lambda$-product of  $\FF$-perfect tree forcing $\mathbb P=\mathbb P(\FF)$ is defined as follows:

The conditions in $\mathbb P(\FF, \lambda)$ are partial functions from $\lambda$ into the collections of all $\FF$-perfect subtrees ${}^{{<\kappa}}\kappa$, satisfying the following conditions:
\begin{description}
\item [{\rm (o)}] $|\dom(p)=\{\iota<\lambda:\, p(\iota)\neq \emptyset\} |\le\kappa$, where for $\iota\in \dom(p)$ we denote by $p(\iota)$ the $\FF$-perfect subtree that $p$ chooses at $\iota$,
\item [{\rm (a)}] for every $\iota\in \dom(p)$, the value of $p(\iota)$ is a condition in $\mathbb P(\mathcal F)$.
\end{description}
The order is given by $p\le q$ if $\dom(p)\subseteq \dom(q)$, while for every $\iota\in\dom(p)$ we have $p(\iota)\le_{\mathbb P(\FF)} q(\iota)$.

\smallskip

{\noindent (2)} For $\zeta<\kappa$ we define $p\le_\zeta q$ if 
\begin{itemize}
\item $p\le q$ and
\item for every $\iota\in \dom(p)$ we have $p(\iota)\le_\zeta^{\mathbb P(\FF)} q(\iota)$. 
\end{itemize}
 \end{definition}
 
\begin{claim}\label{obs:generic-productPF} The generic object added by $\mathbb P(\FF, \lambda)$ is a sequence of $\lambda$-many distinct $\mathbb P(\mathcal F)$-generic
functions.
\end{claim}

\begin{proof} Let $G$ be a $\mathbb P(\FF, \lambda)$-generic filter and for $\iota<\lambda$ let 
\[
G_\iota=\{p(\iota):\,\iota<\lambda, p\in G, \iota\in \dom(p)\}.
\]
We claim that each $G_\iota$ is a $\mathbb P(\FF)$-generic filter over $\mathbf V$. Let us fix $\iota$ and let us first show that $G_\iota$ is a filter.
If we have $p(\iota), q(\iota)\in G_\iota$ for some $p,q\in G$, then there is $r\in G$ with $p,q\le r$. Hence $p(\iota), q(\iota)\le r(\iota)$, and $r(\iota)\in G_\iota$. If $p(\iota)\ge_{\mathbb P(\FF)} p'$ and $p(\iota)$ in $G_\iota$, then by defining 
\begin{equation*}
q(j)=\begin{cases}
p(j) &\mbox{if $j\in \dom(p)\setminus \{\iota\}$},\\
p'  &\mbox{if $j=\iota$}
\end{cases}
\end{equation*}
we obtain $q\in \mathbb P(\FF, \lambda)$ and $q\le p$. So $q\in G$ and hence in particular $q(\iota)=p'\in G_\iota$. These observations show that $G_\iota$ is a filter in $\mathbb P(\FF)$. Let us now show that it is $\mathbb P(\FF)$-generic over $\mathbf V$.

Suppose that $\DD\in \mathbf V$ is a dense subset of $\mathbb P(\FF)$ and let 
\[
\EE=\{p\in \mathbb P(\FF, \lambda):\,p(\iota)\in \DD \}.
\]
Then $\EE\in \mathbf V$. We claim that $\EE$ is a dense subset of $\mathbb P(\FF, \lambda)$, so let $p\in \mathbb P(\FF, \lambda)$ be given. We have that $p(\iota)\in \mathbb P(\FF)$
and hence $p(\iota)$ has an extension $p' \in \DD$. Define $q\in \mathbb P(\FF, \lambda)$ by
\begin{equation*}
q(j)=\begin{cases}
p(j) &\mbox{if $j\in \dom(p)\setminus \{\iota\}$},\\
p'  &\mbox{if $j=\iota$}.
\end{cases}
\end{equation*}
Then $q$ is an extension of $p$ and $q\in \EE$. This argument shows that $\EE$ is dense in $\mathbb P(\FF, \lambda)$, and hence there is
$p\in G\cap \EE$. Then $p(\iota)\in G_\iota\cap \DD$. This shows that $G_\iota$ intersects all dense subsets of $\mathbb P(\FF)$ which are in $\mathbf V$. We have therefore shown that $G_\iota$ is a 
$\mathbb P(\FF)$-generic filter over $\mathbf V$. Let $f_\iota$ denote $f[G_\iota]$, which as shown in Lemma \ref{lem:generic-kappa-function} is a function from $\kappa$ to $\kappa$.

It remains to show that for $\iota\neq j\in \lambda$, the functions $f_\iota$ and $f_j$ are distinct. Let 
\[
\DD_{\iota,j}=\{p\in \mathbb P(\FF, \lambda):\,\iota, j\in \dom(p)\mbox{ and }(\exists \alpha\in \dom(p(\iota)\cap \dom(p(j)) [p(\iota)(\alpha)\neq p(j)(\alpha)]\}.
\]
Then $\DD_{\iota,j}\in \mathbf V$, it is a subset of $\mathbb P(\FF, \lambda)$ and we claim that it is dense. So, given $p\in \mathbb P(\FF, \lambda)$,
first we may extend it so that $\iota,j\in \dom(p)$. Let
$\alpha=\sup (\dom(p(\iota))\cup \dom(p(j)))+1$. Let
$q$ be defined by 
\begin{equation*}
q(k)=\begin{cases}
p(k) &\mbox{if $k\in \dom(p)\setminus \{\iota, j\}$},\\
p(\iota)\cup \{(\alpha,0)\}  &\mbox{if $k=\iota$},\\
p(j)\cup \{(\alpha,1)\}  &\mbox{if $k=j$}.
\end{cases}
\end{equation*}
Then $q\ge p$ and $q\in \DD_{\iota,j}$. If $r\in G\cap \DD_{\iota,j}$, as exemplified by $\alpha$, then $r_\iota(\alpha)\neq r_j(\alpha)$ and both are defined.
Hence $f_\iota(\alpha)= r_\iota(\alpha)\neq r_j(\alpha)=f_j(\alpha)$, showing that $f\iota\neq f_j$.
$\eop_{\ref{obs:generic-productPF}}$
\end{proof}

\begin{notation}\label{union-not} Suppose that $\bar{p}=\langle p_i:\;i<i^*\rangle$ for some $i^\ast<\kappa$
is an increasing sequence of conditions in $\mathbb P(\FF, \lambda)$.
By $\bigcup \bar{p}$ we mean the function $q$ defined by letting $\dom(q)=\bigcup_{i<i^\ast} \dom(p_i)$ while for every
$\alpha\in \dom(q)$ we have $q(\alpha)=\bigcap_{i<i^\ast} p_i(\alpha)$.
\end{notation}

\begin{lemma}\label{P:closurelambda} The product $\mathbb P(\FF, \lambda)$ is $(<\kappa)$-closed, and moreover, every increasing
sequence $\bar{p}=\langle p_i:\;i<i^*\rangle$ of conditions in $\mathbb P(\FF, \lambda)$ where $i^*<\kappa$, has a lub which is 
$\bigcup \bar{p}$.
\end{lemma}

\begin{proof} This follows from Lemma \ref{when-closed}. 
$\eop_{\ref{P:closurelambda}}$ 
\end{proof}

\begin{lemma}[{\bf Fusion Lemma}]\label{Sacks:FusionP} Suppose that $0<\delta\le \kappa$ is a limit ordinal and $\bar{p}=\langle p_\zeta:\,\zeta<\delta\rangle$ is such that $p_\zeta\le_{\zeta+1}  p_{\zeta+1}$ for each $\zeta$, while for limit ordinals $0<\varepsilon <\delta$ and $\iota\in \dom(p_\varepsilon)$ we have $p_\varepsilon(\iota)=\bigcap_{\zeta<\varepsilon} p_\zeta(\iota)$.
Then $q=\bigcup \bar{p}$ is a condition in $\mathbb P(\mathcal F, \lambda)$ such that for every $\zeta<\delta$ we have $p_\zeta\le_\zeta q$, and in addition,
$q$ is the least condition with these properties.
\end{lemma}

\begin{proof} Since $\delta\le\kappa$ and $|\dom(p_\zeta)|\le\kappa$ for all $\zeta$, we 
have $|\dom(q)|\le\kappa$. For each $\iota\in \dom(q)$ we have that 
$\langle p_\zeta(\iota):\,\zeta<\delta\rangle$ is a fusion sequence in $\mathbb P(\FF)$ and hence by Lemma \ref{Sacks:Fusion} we have that
$q(\iota)=\bigcap_{\zeta<\delta} p_\zeta(\iota)$ is an upper bound of that sequence and the least such. Hence $q$ is a lub  of $\bar{p}$ and we similarly have that $p_\zeta\le_\zeta q$ for any $\zeta<\delta$.
$\eop_{\ref{Sacks:FusionP}}$
\end{proof}

The countable support product of Perfect Set Forcing preserves $\aleph_1$. This is an argument of Baumgartner from \cite[\S1]{zbMATH03926906} and is also well described in Jech \cite[Theorem I 5.12]{Jec86}.
The main idea of the {\em Mastering $\kappa$ argument} is partial fusion. A partial fusion sequence in the product is built through an increasing sequence $\langle p_n: \, n<\omega \rangle$ where $p_n\le_{F_n, n} p_{n+1}$ and the the $F_n$s are finite sets with the union equal to $\bigcup_{n<\omega} \dom(p_n)$. At the step $n$ we consider all the $n$-fronts of $p_n$ (finitely many) and build and increasing sequence of conditions where we have fused all the possibilities for the forced value of the name of the ordinal we are dealing with.  The success of this proof depends on the fact that the number of fronts to consider is finite and the number of possible extensions of length $n+1$ is finite. In our case, or in the case of Grigorieff forcing for $\kappa=\aleph_0$, the number of fronts is $\kappa$, since we have filter-many extensions of every stem to consider. While we could still do the fusion on
the coordinates in $F_n$, the remaining parts of the conditions to fuse are not increasing in a fusion order, but only in the original order $\le$, so they might not have a common upper bound.

The following shows that $\mathbb P(\FF, \kappa)$ collapses $\kappa^+$. 

\begin{theorem}\label{claim:collapse} Suppose that $2^\kappa=\kappa^+$ in the ground model. Then $\mathbb P(\FF, \kappa)$ collapses $\kappa^+$.
\end{theorem}

\begin{proof} The proof is modelled after the proof of \cite[Proposition 5.17]{Jec86}, credited to Baumgartner. 

We recall that $\FF$ is a $(<\kappa)$-complete filter on $\kappa$ in $\V$ and that a condition in $\mathbb P(\FF, \kappa)$ is of the form $\langle p(\iota):\,\iota<\kappa\rangle$, where each $p(\iota)$ is a condition in $\mathbb P(\FF)$
(see Definition \ref{gSacks:def}). As shown in Claim \ref{obs:generic-productPF}, the generic object added by this forcing is a sequence $\langle f_\iota:\,\iota<\kappa\rangle$ where each $f_\iota$
is $\mathbb P(\FF)$-generic. We are going to define a $\kappa$-sequence of $\mathbb P(\FF, \kappa)$-names 
$\langle \name{h}_\alpha:\,\alpha <\kappa\rangle$ for functions from $\kappa$ to $\kappa$, as follows. Given $\alpha<\kappa$, we define $\name{h}_\alpha(\iota)$
for $\iota<\kappa$, by induction on $\iota$ and using the canonical names for $\langle \name{f}_\iota:\,\iota<\kappa\rangle$ as a parameter.

We let $\name{h}_\alpha(0)=\name{f}_0(\alpha)$. For an ordinal of the from $\iota+1$ for $\iota<\kappa$ we let 
$\name{h}_\alpha(\iota+1)=\name{f}_{\iota+1}(\name{h}_\alpha(\iota))$. For $\iota$ limit $>0$ we let $\name{h}_\alpha(\iota)=\name{f}_{\iota}(\sup_{j<\iota} \name{h}_\alpha(j))$. 

Using $2^\kappa=\kappa^+$, we let $\{e_j:\,j<\kappa^+\}$ enumerate ${}^\kappa \kappa$ in $\V$. We shall show that the set of conditions $q$ in $\mathbb P(\FF, \kappa)$ 
which force that 
\[
\sup (\{j<(\kappa^+)^\V:\,(\exists \alpha<\kappa:\,\name{h}_\alpha= e_j\})=(\kappa^+)^\V
\]
is dense, hence showing that $\cf^{\V[G]}(\kappa^+)^\V\le \kappa$, as claimed. So, let $\langle p(\iota):\,\iota<\kappa\rangle$ be a given a condition and let $j_0<\kappa^+$ be given. We shall find $q\ge p$ and $j\ge j_0$ 
such that $q\forces ``\name{h}_\alpha=e_j"$ for some $\alpha<\kappa$. In fact, we shall simply assure that $q$ forces that $\name{h}_\alpha$
is not among $\{e_j:\,j< j_0\}$, yet that it decides that $\name{h}_\alpha$ is in the ground model.

Let us reenumerate, possibly with repetitions, $\{e_j:\,j< j_0\}$ as $\{g_\iota:\,\iota<\kappa\}$. We define $q=\langle q(\iota):\,\iota<\kappa\rangle$
by choosing $q(\iota) \in \mathbb P$ by induction on $\iota<\kappa$. We shall use an easily established property of $\mathbb P(\FF)$.

\begin{observation}\label{obs:aslarge} Given a condition $r\in \mathbb P(\FF)$, a function $g:\;\kappa\to\kappa$ and $\alpha<\kappa$. There is a $r'\ge r$
which forces that for some $\beta> \alpha$ we have $f_{\name{G}}(\beta)> g(\beta)$.
\end{observation}

[This is true because we can find a splitting node $s$ of $r$ of height $\beta$ above $\alpha$. Since $s$ splits into $\FF$-many values, there is $i>g(\beta)$
such that $s'=s\,\frown (\beta,i)$ is a node of $r$. Then it suffices to let $r'=r_{s'}$.]

\smallskip

To continue the proof, let $\alpha>\lg(s[p(0)])$. We shall make sure that the condition $q$ forces that $\name{h}_\alpha$
is in the ground model and is different than any $g_\iota$ for $\iota<\kappa$. Let $q_0\ge_ {\mathbb P(\FF)} p(0)$ be such that $\lg (s[q_0])>\alpha$ and such that we have $s[q_0](\alpha)> g_0(\alpha)$. Hence, if $q_0$ will be $q(0)$, then $q$ will decide the value of 
$\name{h}_\alpha(0)=\name{f}(0)(\alpha)$ and also make sure that $\name{h}_\alpha\neq g_0$.
 For an ordinal of the from $\iota+1$ for $\iota<\kappa$, we assume that $q(\iota)\in \mathbb P(\FF)$ was chosen so that any condition $q\in \mathbb P(\FF, \kappa)$ with the initial 
 part $\langle q(0), \ldots q(\iota)\rangle$ decides the value of $\name{h}_\alpha(\iota)$, which we shall denote by $\alpha_\iota$.
 Then we let $q({\iota +1})\ge p({\iota+1})$ be such that $\lg(s[q({\iota +1}])> \alpha_\iota$ and that $s[q({\iota +1})](\alpha_\iota)> g_\iota(\alpha_\iota)$.
 So we have succeeded to continue the inductive hypothesis, as indeed any $q\in \mathbb P(\FF, \kappa)$ which starts with $\langle q(0), \ldots q({\iota+1})\rangle$ 
 decides the value of $\name{h}_\alpha(\iota+1)$, since $s[q({\iota +1)}](\alpha_\iota)$ is well defined.

 At limit stages $\iota$, we assume that for every $j<\iota$ we have chosen $q (j) \in \mathbb P(\FF)$ so that any condition $q\in \mathbb P(\FF, \kappa)$ with $\langle q(0), \ldots q(j)\rangle$ as an initial part decides the value $\alpha_j=\name{h}_\alpha(j)$. Let $\alpha_\iota=\sup_{j<\iota}\alpha_j$ and
 choose $q(\iota)\in \mathbb P(\FF)$ so that $\lg(s[q(\iota)])>\alpha_\iota$ and that $s[q_{\iota}](\alpha_\iota)> g_\iota(\alpha_\iota)$. 

At the end, the condition $q=\langle q(\iota):\,\iota<\kappa\rangle$ is as required.
$\eop_{\ref{claim:collapse}}$
\end{proof}

\section{Iteration with Support of Size $\le\kappa$ of Property $B(\kappa)$ Forcing}\label{subsec:iteration-general} We assume that $\kappa$ 
satisfies $\kappa^{<\kappa}=\kappa$. We shall prove a conditional 
iteration theorem for a property $B(\kappa)$ forcing, with three requirements on individual forcings - essentially meaning that we have property $B(\kappa)$ in a 
definable way- and the fourth requirement
on the iteration itself. The conditions in question are:
\begin{enumerate}
\item a strong $(<\kappa)$-closure of the individual forcings, which is  a closure witnessed by a function (to be explained in detail below), 
\item a strong fusion of the individual forcings, which is a fusion witnessed by a function (to be explained in detail below),
\item a strengthening of Mastering $\kappa$ of the individual forcings, which we shall call {\em $B$-$\kappa$-properness}, and
\item the property of the iteration which we shall call {\em amalgamation}. This corresponds to the idea of being able to make decisions in $\le$ 
by extensions which are in addition $\le_\zeta$ on any given set of $<\kappa$-many coordinates.
\end{enumerate}
While the first three properties are iterable, given the last one, the last property is not iterable and hence we have to study separately how to achieve that property by the properties of the individual forcing and the iteration itself. In \S\ref{sec:whichones}, we prove that an iteration of forcing of type ${\mathbb P}(\FF)$
with supports of size $\le\kappa$ satisfies the conditions above.

We now give the definitions of the properties we consider, in the order announced, starting with the strong closure.

\begin{definition}\label{def:strongclosure} A {\em strongly $(<\kappa)$-closed forcing notion} is a pair $({\mathbb P},F)$ where ${\mathbb P}$
is a forcing and
$F$ is a function which acts on increasing sequences of length $<\kappa$ of conditions in ${\mathbb P}$, so that for every 
such sequence $\bar{p}=\langle p_i:\,i<i^\ast\rangle$ where $i^\ast<\kappa$ we have that $F(\bar{p})$ is a common upper bound to $\bar{p}$.

In such circumstances we say that $F$ {\em witnesses} the strong $(<\kappa)$-closure of ${\mathbb P}$. If no confusion arises, we often just say that
${\mathbb P}$ is strongly $(<\kappa)$-closed when we mean that we have fixed a function $F$ such that $({\mathbb P},F)$ is a strongly $(<\kappa)$-closed forcing notion.
\end{definition}

\begin{example}\label{ex-strongclosure}
(1) Lemma \ref{when-closed} proved that the perfect set forcing ${\mathbb P}(\FF)$ with respect to a $(<\kappa)$-complete filter is strongly 
$(<\kappa)$-closed, where for the value of the function $F$ applied to an increasing sequence $\langle p_i:\,i< i^\ast\rangle$ we take the intersection
$\bigcap_{i<i^\ast} p_i$. 

{\noindent (2)} Any product of 
strongly $(<\kappa)$-closed forcing is strongly $(<\kappa)$-closed, as we can define the function witnessing the strong $(<\kappa)$-closure co-ordinatwise.
\end{example}

Next is the property of strong fusion.

\begin{definition}\label{def:strongfusion} We say that a forcing notion ${\mathbb P}$ together with a sequence
$\langle \le_\zeta:\,\zeta<\kappa\rangle$ of orders witnessing fusion satisfies {\em strong fusion}, if there is a function 
$t:\,{}^{\le \kappa} {\mathbb P}\to {\mathbb P}$
with the properties described below.

For any $\delta\le \kappa$ a limit ordinal, $\zeta^\ast<\delta$ and a sequence $\langle p_\zeta:\,\zeta\in [\zeta^\ast, \delta)\rangle$ such that:
\begin{itemize}
\item $p_\zeta \le_\zeta p_{\zeta+1}$ for all $\zeta \in [\zeta^\ast, \delta)$,
\item for any $\varepsilon\in [\zeta^\ast, \delta)$ a limit ordinal, the condition $p_\varepsilon$ is the value of $t(\langle p_\zeta:\,\zeta\in [\zeta^\ast, \varepsilon)\rangle)$,
\end{itemize}
the condition $t(\langle  p_\zeta:\,\zeta\in [\zeta^\ast, \delta)\rangle)$ is the fusion limit of $\langle p_\zeta:\,\zeta \in [\zeta^\ast, \delta)\rangle$.

\end{definition}

\begin{example}\label{ex-strongfusion} (1) Lemma \ref{Sacks:Fusion} shows that the forcing $\mathbb P(\FF)$ has strong fusion, where 
\[
t(\langle  p_\zeta:\,\zeta\in [\zeta^\ast, \delta)\rangle)=\bigcap_{\zeta\in [\zeta^\ast, \delta)} p_\zeta.
\]

{\noindent (2)} Lemma \ref{Sacks:FusionP} showed that we have strong fusion for the product of $\mathbb P(\FF)$ forcings. In general, the same argument shows that strong fusion carries to any product of strong fusion forcing, as we can define the witnessing functions coordinatwise.

\end{example}

We shall use the following strengthening (see Lemma \ref{implication-properness}) of {\em Mastering $\kappa$}. 

In the following we use $\chi$ for a regular cardinal large enough so that all the objects mentioned in our proof are elements of $\HH(\chi)$. 
%We also fix a well order $\prec^\ast$ of $\HH(\chi)$.

\begin{definition}\label{def:B-kappa-properness} We say that a forcing notion $\mathbb P$ is {\em $B$-$\kappa$-proper} if
there exists a sequence $\langle \leq_\zeta:\,\zeta<\kappa\rangle$ of partial orders on $P$, satisfying 
\begin{enumerate}
\item  \(\leq_0 = \leq\), and $\xi<\zeta\implies \le_\xi\supseteq \le_\zeta$, 
    \item  for $\delta$ limit $\in (0,\kappa)$, we have $\le_\delta=\bigcap_{\zeta<\delta} \le_\zeta$,
    \item there is some $y\in H(\chi)$ such that
\begin{itemize}
\item whenever $p\in \mathbb P$, $\zeta^\ast<\kappa$, $N\prec H(\chi)$ with $|N|=\kappa$ and $N^{<\kappa}\subseteq N$ (footnote\footnote{Note that these requirements imply that $N$ contains all ordinals $\alpha<\kappa$, i.e. $\kappa\subseteq N$. Namely, if $\alpha<\kappa$ were to be the first ordinal $<\kappa$ such that $\alpha\notin N$, then clearly $\alpha\neq 0$ by $N\prec H(\chi)$ and the definability of 0, and $\alpha\neq \beta+1$ for any $\beta$- for similar reasons, since $\beta+1$ is definable from $\beta$. Hence
$\alpha>0$ is a limit ordinal. So $\alpha$ is the limit of the sequence $\langle \beta:\beta<\alpha\rangle$. Every element of this sequence is in $N$ and hence the sequence itself is in $N$ by $N^{<\kappa}\subseteq N$. Being the limit of a sequence is definable from that sequence as a parameter, so by the elementarity of $N$ it follows that $\alpha$ in $N$. A contradiction.})

 are such that $p, y,  \mathbb P,\langle \le_\zeta:\,\zeta<\kappa\rangle \in N$,
\end{itemize}
then there is $q\ge_{\zeta^\ast} p$ which is $\mathbb P$-generic over $N$.
\end{enumerate}
\end{definition}

Notice that the assumptions on $N$ above imply that $\kappa$, definable as the length of $\langle \le_\zeta:\,\zeta<\kappa\rangle$, is an element of $N$
and then also a subset of $N$. Then by $N^{<\kappa}\subseteq N$ it follows that $\kappa^{<\kappa}\subseteq N$.

\begin{lemma}\label{implication-properness} Suppose that $\mathbb P$ is a forcing notion satisfying 
the first two items in the definition of property $B(\kappa)$,
Definition \ref{def:propertyB}, as witnessed by
$\langle \le_\zeta:\,\zeta<\kappa\rangle$, and that it is $B$-$\kappa$-proper. Then it satisfies {\em Mastering $\kappa$}.
\end{lemma}

\begin{proof} 
Suppose that $p\in \mathbb P$ and $p\forces``\name{\tau}\mbox{ is an ordinal}"$, while $\zeta<\kappa$. Then we can find $N$ with the properties such as in the definition 
above and $p\in N$. Let $q\ge_\zeta p$ be $\mathbb P$-generic over $N$. Then $q\forces {``}\name{\tau}\in N"$ (see  \cite[Fact 3.2]{Goldsterniteration} for
various equivalent definitions of being a generic condition, including this one) and hence $N$ is a set of size $\le\kappa$ which we can use for $X$ 
in the definition of {\em Mastering $\kappa$}. In particular, $\mathbb P$ preserves $\kappa^+$.
$\eop_{\ref{implication-properness}}$
\end{proof}

An incorrect claim about the equivalence of properness and {\em Mastering $\aleph_0$} appeared as \cite[III, Lemma 2.6]{zbMATH03779315}. See more on the history of that claim in \S\ref{sec:history}. However, it is true that Axiom A forcing is proper, as proved by Baumgartner in 
\cite[Th 2.4]{Ba_AppPFA}.
A similar proof, which we shall give here, shows that any forcing satisfying property $B(\kappa)$ with strong fusion is,  $B$-$\kappa$-proper.

\begin{theorem}\label{th:Aimpliespr} Any forcing satisfying property $B(\kappa)$ with a strong fusion, is $B$-$\kappa$-proper.
\end{theorem}

\begin{proof} Let ${\mathbb P}$ be a forcing notion satisfying property $B(\kappa)$ with a strong fusion, let $\langle \le_\zeta:\,\zeta<\kappa\rangle$ some orders witnessing $B(\kappa)$ with a strong fusion $t$, and let $p\in {\mathbb P}$ and $\zeta^\ast<\kappa$. Set $y=t$ and
suppose that $N\prec \HH(\chi)$ is of size $\kappa$ with ${N}^{<\kappa} \subseteq N$ and ${\mathbb P}, y, p, \langle \le_\zeta:\,\zeta<\kappa\rangle \in N$.
Let $\langle \DD_\alpha:\,\alpha<\kappa\rangle$ enumerate all open dense subsets of ${\mathbb P}$ which are in $N$. By induction on
$\alpha<\kappa$ we construct a sequence $\langle p_\alpha: \,\alpha<\kappa\rangle$ so that:
\begin{itemize}
\item $p_0=p$,
\item $p_{\alpha+1}\ge_\alpha p_\alpha$ is such that $\DD_\alpha$ has a subset $\EE_\alpha$ of size $\le\kappa$ which is dense above $p_{\alpha+1}$,
\item at limit stages $\delta\in (0,\kappa)$ of the induction the condition $p_\delta$ is the strong fusion 
\[
t(\langle p_\alpha:\,\alpha<\delta\rangle).
\]
\end{itemize}
The possibility to run this induction is given by the definition of property $B(\kappa)$ with strong fusion $t$. By elementarity, we can choose the sets $\EE_\alpha\in N$. 
Since $|\EE_\alpha|\le\kappa$ and $\kappa\in N$, we can find a surjection $f:\,\kappa\to \EE_\alpha$ with $f\in N$. Since $\kappa\subseteq N$ we
have that for every $\beta<\kappa$ the value $f(\beta)\in N$. Hence, we conclude that $\EE_\alpha\subseteq  N$.

At the end of this inductive definition we can define $q=t(\langle p_\alpha:\,\alpha<\kappa\rangle)$. In particular we have $p\le_\zeta q$ for any $\zeta$, so $p\le_{\zeta^\ast} q$. 

We claim that the condition $q$ is ${\mathbb P}$-generic over $N$. To see that, let $\DD$ be an open dense subset of ${\mathbb P}$ which is in $N$. We need to 
show that $\DD\cap N$ is predense above $q$, so suppose that $r\ge q$. Let $\alpha$ be such that $\DD=\DD_\alpha$. Since we have $p_{\alpha+1}\le_{\alpha+1} q\le r$, we know that there is $s\in \EE_\alpha$ with $s\ge r$, by the choice of $\EE_\alpha$. Since $\EE_\alpha\subseteq  N$
we have $s\in N$. Hence $s\in \DD \cap N$ is compatible with $r$.
$\eop_{\ref{th:Aimpliespr}}$
\end{proof}

Hence, in theory we could have done the whole proof of the iteration theorem without introducing the notion of $B$-$\kappa$-properness, but the iteration is easier to explain in terms of $B$-$\kappa$-properness and this discussion also allows us to better see the fine differences between the preservation of cardinals and properness. 

Finally the last, non-iterable, property and a crucial one is given by the following definition.

\begin{definition}\label{def:amalgamation} An iteration of property $B(\kappa)$ forcing $\langle ({\mathbb P}_\alpha, \name{Q}_\alpha):\,\alpha<\alpha^*\rangle$ is said to have {\em amalgamation} if for every $\alpha<\alpha^\ast$, for every $C\in [\alpha]^{<\kappa}$, for every $\zeta<\kappa$, for every dense open set 
$\DD\subseteq {\mathbb P}_\alpha$ and every $p\in {\mathbb P}_\alpha$, there is $q$ such that:
\begin{itemize}
\item $q\in \DD$,
\item $q\ge p$,
\item for every $\beta\in C$, we have $q\rest\beta\forces_{\mathbb P_\beta}``q(\beta)\name{\ge}^{\name{Q}_\beta}_{\zeta}\, p(\beta)"$.
\end{itemize}
\end{definition}

\begin{observation}\label{label:propimplamalgamation} It follows from the definitions that the property of being B-$\kappa$-proper {\em for the iteration} implies the 
property of amalgamation, but this is not our assumption. The assumption is being B-$\kappa$-proper for each individual forcing. The main point will be that with that and the amalgamation of the iteration, we can prove that the limit of the iteration is B-$\kappa$-proper.
\end{observation}

Before we state and prove the iteration theorem given the properties defined, we recall  for completeness the classical definition of 
the iteration with supports of size $\le\kappa$. See \S \ref{sec:history} of why we feel that we should state exactly the definition we use.

\begin{definition}\label{def:mixed-iteration} By induction on the ordinal $\alpha^*$ we define what is a {\em an iteration of length $\alpha^*$ with supports of size $\le\kappa$ of 
$B$-$\kappa$-proper strongly $(<\kappa)$-closed with strong fusion forcings}, and what is meant by {\em the limit ${\mathbb P}_{\alpha^*}$ of the iteration}. 

If $\alpha^*=0$, the iteration is a sequence of length 1 whose only entry is the trivial forcing notion, which is also the limit of the iteration. For
$\alpha^\ast>0$, the iteration is a sequence of the form $\langle {\mathbb P}_\alpha, (\name{Q}_\alpha, \name{F}_\alpha):\,\alpha< \alpha^*\rangle$ such that:
\begin{itemize}
\item For $\alpha<\alpha^*$, the forcing notion ${\mathbb P}_\alpha$ is the limit of an iteration with supports of size $\le\kappa$ of an iteration $\langle {\mathbb P}_\beta, (\name{Q}_\beta, \name{F}_\beta):\,\beta<\alpha\rangle$ of $B$-$\kappa$-proper strongly $(<\kappa)$-closed forcings with strong fusion.
\item For every $\alpha<\alpha^*$, it is forced by ${\mathbb P}_\alpha$ that $\name{Q}_\alpha$ is a forcing notion and that:
\begin{itemize}
\item $\name{Q}_\alpha$ is strongly $(<\kappa)$-closed, as witnessed by a function $\name{F}_\alpha$ (in other words, the pair $(\name{Q}_\alpha, \name{F}_\alpha)$ is a strongly $(<\kappa)$-closed forcing notion), 
\item $\name{Q}_\alpha$ has the strong fusion property, as witnessed by $\langle \name{\le}_\zeta^\alpha:\,\zeta<\kappa\rangle$ and a function $\name{t}_\alpha$ and
\item $(\name{Q}_\alpha, \langle \name{\le}_\zeta^\alpha:\,\zeta<\kappa\rangle)$ is $B$-$\kappa$-proper.
\end{itemize}
\item If $\alpha^\ast=\alpha+1$ then ${\mathbb P}_{\alpha^*}={\mathbb P}_{\alpha} \ast \name{Q}_\alpha$.
\item Elements of ${\mathbb P}_{\alpha^*}$ are sequences of the form $p=\langle p(\alpha):\,\alpha<\alpha^\ast\rangle$ satisfying the following:\begin{enumerate}
\item for every $\alpha<\alpha^*$, the sequence $\langle p(\beta):\,\beta<\alpha\rangle$ is in ${\mathbb P}_\alpha$,
\item each $p(\alpha)$ is a canonical ${\mathbb P}_\alpha$-name for a condition in $\name{Q}_\alpha$,
\item ${\rm supt}(p)=\{\alpha<\alpha^\ast:\,\neg (p\rest\alpha\forces_{{\mathbb P}_{\alpha}} ``p(\alpha)=\emptyset_{\name{Q}_\alpha}")\}$ is a set of size $\le\kappa$.
\end{enumerate}
\item The order $\le_{\alpha^\ast}$ on ${\mathbb P}_{\alpha^\ast}$ is given by $p\le q$ iff the following hold:
\begin{enumerate}
\item if $\alpha^\ast=\alpha+1$, then $(p\rest\alpha, p(\alpha))\le (q\rest\alpha, q(\alpha))$ in the two step iteration ${\mathbb P}_\alpha\ast \name{Q}_\alpha$,
\item if $\alpha^\ast$ is a limit, then for all $\alpha<\alpha^\ast$ we have $p\rest\alpha\le q\rest \alpha$.
\end{enumerate}
 \end{itemize}
 \end{definition}

\begin{remark}\label{F-or-not-F} In all precision, we should have included the orders $\langle\name{\le}_\zeta^\alpha:\,\zeta<\kappa\rangle$ and the witnesses 
$t_\alpha$ in our notation. We have felt that it would make the notation very cumbersome, but have chosen to emphasise the existence of the witness $F_\alpha$ nevertheless. We shall omit it whenever the demand for a witness is clear from what has already been said.

To further simplify the expression, we may speak of a $\le\kappa$-support iteration of individual forcings that have the properties of 
 strong $(<\kappa)$-closure,  strong fusion and $B$-$\kappa$-properness when meaning the entire definition above.

\end{remark}

\begin{theorem}[General Iteration Theorem]\label{th:iteration} Let $\kappa^{<\kappa}=\kappa$. Suppose that 
\[
\langle {\mathbb P}_\alpha, 
(\name{Q}_\alpha, \name{F}_\alpha):\,\alpha<\alpha^*\rangle
\]
is a $\le\kappa$-support iteration of individual forcings that
\begin{itemize}
\item  are strongly $(<\kappa)$-closed, 
\item have strong fusion and
\item are $B$-$\kappa$-proper, 
\end{itemize}
and further suppose that the iteration $\langle ({\mathbb P}_\alpha, \name{Q}_\alpha):\,\alpha<\alpha^*\rangle$ has amalgamation.
Then the limit ${\mathbb P}_{\alpha^*}$ of the iteration has the following properties:
\begin{enumerate}
\item it is strongly $(<\kappa)$-closed and
\item it is $B$-$\kappa$-proper.
\end{enumerate}
In  particular, ${\mathbb P}_{\alpha^*}$ preserves cardinals up to and including $\kappa^+$.
\end{theorem}

We explain the main ideas of the proof, comparing it with the known literature. 
The proof relies on several lemmas, the most important of which is the Fusion Lemma, presented as Lemma \ref{lem:fusion-lem}. It is analogous to Baumgartner's \cite[Lemma 7.2]{Ba} used to handle preservation of 
$\aleph_1$ in a countable support iteration of Axiom A forcing, but it has some new ideas. The main point is to construct two increasing sequences of conditions, one (the $q$s) that increases in length with the initial segments frozen, and the other, the sequence of $p$s whose growth is controlled by a fusion. The use of elementary submodels in the
proof of Lemma \ref{cor:preser1} comes from Shelah's work on
the preservation of properness by countable support iterations \cite{Sh_P}[III, Th 3.2]. As in these two proofs, the idea of Fusion lets us pass the limits
of cofinality $<\kappa$. The new case here are limit cardinals of cofinality $\kappa$. 

It is in the proof of Lemma \ref{cor:preser1} that we use the assumption of the iteration having amalgamation. This is a main point.

\begin{proof} For every $\alpha<\alpha^\ast$, let $\name{F}_\alpha$ be a ${\mathbb P}_{\alpha}$-name for a function witnessing that 
$\name{Q}_\alpha$ is strongly $(<\kappa)$-closed and let $\langle \name{\le}^\alpha_\zeta:\,\zeta<\kappa\rangle$ and $\name{t}_\alpha$
be ${\mathbb P}_{\alpha}$-names that witness the strong fusion of $\name{Q}_\alpha$.

\smallskip

{\noindent (1)} We shall prove by induction on $\alpha^\ast$ that ${\mathbb P}_{\alpha^\ast}$ is strongly $(<\kappa)$-closed as witnessed by a function $h_{\alpha^*}$.

Suppose that $\langle p_i:\,i<i^*\rangle$ is an increasing sequence in ${\mathbb P}_{\alpha^\ast}$ and $i^\ast<\kappa$. By induction on 
$\alpha<\alpha^\ast$ we choose an upper bound $q_\alpha=h_\alpha(\langle p_i\rest\alpha:\,i<i^*\rangle)$, and we claim that we can do it so that 
for every $\beta<\alpha$ we have $q_\beta= q_\alpha\rest\beta$.

The case \underline{$\alpha=0$} is trivial. For the case \underline{$\alpha=\beta+1$}, this is like the well-known proof that the iteration of two $(<\kappa)$-closed forcings is $(<\kappa)$-closed, but we need to make that property functional. By the induction hypothesis, ${\mathbb P}_\beta$ is a strongly 
$(<\kappa)$-closed forcing, as witnessed by a function $h_\beta$.

Hence we may define $q_\beta=h_\beta(\langle p_i\rest\beta:\,i<i^*\rangle)$, since the
sequence $\langle p_i\rest\beta:\,i<i^*\rangle$ is increasing. Then $q_\beta$ is a common upper bound for $\langle p_i\rest\beta:\,i<i^*\rangle$.
We then note that $q_\beta$ forces that 
$\langle p_i(\beta):\,i<i^\ast\rangle$ is increasing in $\name{Q}_\beta$ (since for each $i$ we have $p_{i+1}\rest\beta\forces_{{\mathbb P}_\beta}`` p_i(\beta)
\le_{\name{Q}_\beta} p_{i+1}(\beta)"$). Therefore we have that $q_\beta$ forces that $\name{q}'_\beta$ defined by $\name{F}_\beta(\langle p_i(\beta):\,i<i^\ast\rangle)$ is a common upper bound to 
$\langle p_i(\beta):\,i<i^\ast\rangle$. Then $q_\alpha=(q_\beta, \name{q}'_\beta)$ is an upper bound for $\langle p_i\rest\beta:\,i<i^*\rangle$ in ${\mathbb P}_\alpha$ and $q_\beta= q_\alpha\rest\beta$. In conclusion, we may define 
\[
h_\alpha( \langle p_i:\,i<i^*\rangle)=
(h_\beta( \langle p_i\rest\beta:\,i<i^*\rangle), \name{F}_\beta (\langle p_i(\beta):\,i<i^\ast\rangle)).
\]

The case \underline{$\alpha>0$ is a limit ordinal of cofinality $\le\kappa$}. We intend to define $q_\alpha$ as the unique condition in ${\mathbb P}_\alpha$ satisfying that  for all $\beta<\alpha$ we have $q_\alpha\rest\beta=q_\beta$. This definition is functional, so once we have
checked that it is correct, we can define $h_\alpha( \langle p_i:\,i<i^*\rangle)=q_\alpha$.

We have to check that the support of such $q_\alpha$ is still of size $\le\kappa$, so let $j^\ast=\cf(\alpha)$, hence $j^\ast\le\kappa$ and let us fix
an increasing sequence $\langle \alpha_j:\,j<j^\ast\rangle$ with $\alpha=\sup_{j<j^\ast}\alpha_j$. By the induction hypothesis we have that
$\bigcup_{\beta<\alpha}\supt(q_\beta)=\bigcup_{j<j^\ast}\supt(q_{\alpha_j})$, so is of size $\le\kappa$. In particular $q_\alpha$ is indeed a condition in 
${\mathbb P}_\alpha$ and, clearly, for $\beta<\alpha$ we have $q_\alpha\rest\beta=q_\beta$. Thus, $q_\alpha$ is a common upper bound
for $\langle p_i:\,i<i^*\rangle$. 

The case \underline{$\alpha$ is a limit ordinal of cofinality $>\kappa$} follows by defining $h_\alpha(\langle p_i:\,i<i^\ast\rangle)$ as follows.
First we find the first $\beta<\alpha$ such that $\bigcup_{i<i^\ast}\supt(p_i)\subseteq \beta$. Such $\beta$ exists, since $\cf(\alpha)>\kappa>i^\ast$.
Then we let 
\[
q_\alpha=h_\alpha(\langle p_i:\,i<i^\ast\rangle)=h_\beta(\langle p_i\rest\beta:\,i<i^\ast\rangle)\!\frown\![\emptyset_{\name{Q}_\gamma})_{\gamma\in [\beta,\alpha)},
\]
where $h_\beta(\langle p_i\rest\beta:\,i<i^\ast\rangle)$ is well-defined by the induction hypothesis. It also follows that $q_\alpha\rest \gamma=q_\gamma$ for
every $\gamma<\alpha$, as required.

\medskip

{\noindent (2)} In this proof we shall use 
Baumgartner's idea of fusion sequences. We need to adjust it to our needs, which we do as follows. Let $\alpha\le \alpha^*$ be arbitrary.

\begin{definition}\label{def:F-n} (1) For $C\in [ \alpha]^{<\kappa}$, for $\zeta<\kappa$ and $p,q\in {\mathbb P}_\alpha$, we define what it means that 
$p\le_{C,\zeta}^\alpha q$, denoted for short $p\le_{C,\zeta} q$:

$p\le_{C,\zeta} q \mbox{ iff }$
\begin{itemize}
\item $p\le q$ and 
\item $(\forall \beta\in C)\,q\rest\beta\forces_{{\mathbb P}_\beta}``p(\beta)\le_\zeta q(\beta)"$.
\end{itemize}

We note that $\le_{C,\zeta}^\alpha$ is a partial order with the least element $\emptyset_{\alpha}$ and that for $\beta\le \alpha\le \alpha^\ast$
we have that $\le_{C\cap\beta,\zeta}^\beta$  is a complete suborder of  $\le_{C,\zeta}^\alpha$. This justifies the notation $p\le_{C,\zeta} q$. 

\medskip

{\noindent (2)} A {\em fusion sequence in} ${\mathbb P}_\alpha$ is a sequence of the form 
$\langle (p_\zeta, C_\zeta):\,\zeta<\delta\rangle$ such that:
\begin{enumerate}
\item $\delta\le \kappa$ is a limit ordinal,
\item $p_\zeta\in  {\mathbb P}_\alpha$ and $p_\zeta \le_{C_\zeta,\zeta} p_{\zeta+1}$ for all $\zeta$,
\item the sequence $\langle C_\zeta:\,\zeta<\delta\rangle$ is an increasing continuous sequence of elements of $[\alpha]^{<\kappa}$,
\item if $\gamma\in C_{\zeta^\ast}\cap \supt(p_{\zeta^\ast})$, where $\zeta^\ast<\delta$ is the first such ordinal,
then 
\begin{equation}\label{eq:suc-fusion}
p_{\zeta^\ast}\rest\gamma\forces _{\mathbb P_\gamma}``\langle p_\zeta(\gamma):\,\zeta\in [\zeta^\ast, \delta)\rangle \mbox{ is a fusion sequence in }
\name{Q}_\gamma",
\end{equation}
\item if $\xi< \delta$ is a limit ordinal and $\beta\in C_\xi$, then 
\[
p_\xi\rest \beta\forces_{\mathbb P_\beta}``p_\xi(\beta)=\name{t}_\beta (\langle p_\zeta(\beta):\,\zeta<\xi \rangle)".
\]
\end{enumerate}
\end{definition}

\begin{lemma}[Fusion Lemma]\label{lem:fusion-lem} Let $\alpha\le \alpha^\ast$. Then there is a function $H_\alpha$ such that the following holds
for any limit ordinal $\delta\le\kappa$:

if $\langle (p_\zeta, C_\zeta):\,\zeta<\delta\rangle$ is a fusion sequence in
${\mathbb P}_\alpha$ such that for all limit $\xi\in (0,\delta)$ we have $p_\xi=H_\alpha(\langle (p_\zeta, C_\zeta):\,\zeta<\xi\rangle)$,
then letting $q=H_\alpha(\langle (p_\zeta, C_\zeta):\,\zeta<\delta\rangle)$, we have that
\begin{itemize}
\item $q\in {\mathbb P}_\alpha$ and
\item $p_\zeta\le_{C_\zeta,\zeta} q$ holds for any $\zeta<\delta$. 
\end{itemize}
\end{lemma}

\begin{proof}[Proof of Lemma \ref{lem:fusion-lem}] The proof is by induction on $\delta$, for all $\alpha$ simultaneously.
We seek to construct $H_\alpha(\langle (p_\zeta, C_\zeta):\,\zeta<\delta\rangle)$, which we denote by $q^\alpha$. 

For $\delta=0$, we let $q^0=\emptyset_{\mathbb P_0}$.

Now suppose that $\delta\in (0,\kappa]$.
We define $q^\beta\in {\mathbb P}_\beta$ for $\beta\le\alpha$ by induction on $\beta$. The requirements of the induction are
\begin{itemize}
\item for any $\zeta<\delta$ we have $q^\beta\ge_{C_\zeta\cap \beta, \zeta} \,p_\zeta\rest\beta$,
\item for $\gamma\le\beta\le \alpha$ we have $q^\beta\rest\gamma =q^\gamma$.
\end{itemize}

\underline{$\beta=0$.} We let $q^0=\emptyset_{{\mathbb P}_0}$.

\underline{$\beta=\gamma+1$.} We suppose that $q^\gamma$ has been defined as required. 
If $\gamma\notin \bigcup_{\zeta<\delta} \supt (p_\zeta)$, we let $q^\beta=(q^\gamma, 
\emptyset_{\name{Q}_\gamma})$. The conclusion follows by the
induction hypothesis.

Now suppose that $\gamma\in  \bigcup_{\zeta<\delta} \supt (p_\zeta)$. Since $\delta$ is a limit ordinal, we can find a minimal $\zeta<\delta$
such that $\gamma\in \supt(p_\zeta)$, call it $\zeta^\ast$. In $V^{{\mathbb P_\gamma}}$ we shall define $q(\gamma)\in Q_\gamma$ using the function $t_\gamma$ on a sequence of the form
$\bar{r}=\langle r_\zeta:\,\zeta<\delta\rangle$, which we define as follows. For $\zeta\le\zeta^\ast$ we define $r_\zeta=p_{\zeta^\ast}(\gamma)$ and for 
$\zeta>\zeta^\ast$ we define $r_\zeta=p_\zeta(\gamma)$. This is a fusion sequence by clause (\ref{eq:suc-fusion}) in the definition of a fusion sequence.

We set $q^\beta=(q^\gamma, \name{t}_\gamma(\bar{r}))$.

\underline{$\beta$ is a limit ordinal of cofinality $\le\kappa$.} We intend to define $q^\beta$ as the unique condition in ${\mathbb P}_\beta$ satisfying that 
for all $\gamma<\beta$ we have $q^\beta\rest\gamma=q^\gamma$. We have to check that the support of such $q$ is still of size $\le\kappa$. Let 
$i^\ast=\cf(\beta)$, so $i^\ast\le\kappa$ is a limit ordinal. Let us fix
an increasing sequence $\langle \beta_i:\,i<i^\ast\le \kappa\rangle$ with $\beta=\sup_{i<i^\ast}\beta_i$. We have by the induction hypothesis
that $\bigcup_{\gamma<\beta}\supt(q^\gamma)=\bigcup_{i<i^\ast}\supt(q^{\beta_i})$, so is of size $\le\kappa$. Therefore $q^\beta$ is a condition in  ${\mathbb P}_\beta$.

We now claim that for any $\zeta<\delta$ we have $q^\beta\ge_{C_\zeta\cap \beta, \zeta} p_\zeta\rest\beta$, so let us fix such $\zeta$.
First let us see that $p_\zeta\rest\beta\le q^\beta$. By the induction hypothesis, we have that for all $\gamma<\beta$ it holds that 
$p_\zeta\rest\gamma\le q^\gamma$. Then $p_\zeta\rest\beta\le q^\beta$ follows by the definition of $q^\beta$.

If $\gamma\in C_\zeta\cap\beta$, by the fact that $i^\ast$ is a limit ordinal, there is $i<i^\ast$ such that $\gamma\in
C_\zeta\cap\beta_i$. By the induction hypothesis, it follows that 
\[
q^{\beta_i}\rest\gamma\forces_{{\mathbb P}_\gamma}``q^{\beta_i}(\gamma)\ge_\zeta p_\zeta(\gamma)".
\]
The conclusion follows, since $q^\beta\rest(\gamma+1)=q^{\beta_i}\rest(\gamma+1)$.

\underline{$\beta$ is a limit ordinal of cofinality $>\kappa$.} This case is handled trivially by the requirement on the size of the supports. It implies that 
there exist $\gamma<\beta$ such that $\bigcup_{\zeta<\delta}\supt(p_\zeta)\subseteq \gamma$. The induction
was done so that for any $\gamma' <\beta'$ with $\gamma',\beta'\in(\gamma,\beta)$ the condition $q^{\beta'}$ satisfies $q^{\beta'}(\gamma')=\emptyset_{Q_{\gamma'}}$.
Therefore, defining $q^\beta$ as the unique condition in ${\mathbb P}_\beta$ satisfying that 
for all $\beta'<\beta$ we have $q^\beta\rest{\beta}'=q^{\beta'}$, fulfils the requirement.
$\eop_{\ref{lem:fusion-lem}}$
\end{proof}

We arrive to the main point, which is where we need to use the assumption on the iteration having amalgamation.

\begin{lemma}\label{cor:preser1} The forcing ${\mathbb P}_{\alpha^\ast}$ is $B$-$\kappa$-proper.
\end{lemma}

\begin{proof}[{\bf Proof of the Lemma}] Let 
\[
y=\langle (P_\alpha, (\name{Q}_\alpha, \name{F}_\alpha)):\,\alpha<\alpha^\ast\rangle, \langle \le_\zeta:\, \zeta<\kappa\rangle, H_{\alpha^\ast} \rangle,
\]
where $\langle \le_\zeta:\, \zeta<\kappa\rangle$ and $H_{\alpha^\ast}$ witness the strong fusion of  ${\mathbb P}_{\alpha^\ast}$. Let  
$N\prec \HH(\chi)$ with $|N|=\kappa$ and $N^{<\kappa}\subseteq N$ be such that ${\mathbb P}_{\alpha^\ast}, y, p \in N$. 

Let $\langle \DD_\zeta:\;\zeta<\kappa\rangle$ be an enumeration of all dense open sets of ${\mathbb P}_{\alpha^\ast}$ which are in $N$. Let $N\cap \alpha^\ast=
\bigcup_{\zeta<\kappa} C_\zeta$ be a continuous non-decreasing union where for each $\zeta$ we have $|C_\zeta|<\kappa$. In particular, each $C_\zeta\in N$.

By induction on $\zeta<\kappa$ we define a fusion sequence
$\langle (p_\zeta, C_{\zeta}):\,\zeta<\kappa\rangle$ such that
\begin{itemize}
\item $p_0=p$, $p_\zeta\in N$, 
\item $p_{\zeta+1}\in \DD_\zeta$,
\item for any limit $0<\delta<\kappa$ we have $p_\delta=H_{\alpha^\ast}(\langle (p_\zeta, C_{\zeta}):\,\zeta<\delta\rangle)$.
\end{itemize}

{\bf Use of amalgamation.}
To achieve the successor stages of the induction, we shall use the property of amalgamation of $\langle P_\alpha, \name{Q}_\alpha:\,\alpha<\alpha^\ast\rangle$ applied within $N$. That property guarantees that 
we can find  $p_{\zeta+1}\in N\cap \DD_\zeta$ with $p_\zeta \le_{C_\zeta, \zeta} p_{\zeta+1}$.

At non-zero limit stages $\delta$ we
observe that the construction assures that the function $H_{\alpha^\ast}$ applies to the sequence $\langle (p_\zeta, C_\zeta):\,\zeta<\delta\rangle$, so we set $p_\delta=H_{\alpha^\ast}(\langle (p_\zeta, C_\zeta):\,\zeta<\delta\rangle)$, which by elementarity is an element of $N$. 

As $p_\zeta\in N$ for every $\zeta$, we have that $\supt(p_\zeta)\subseteq N \cap \alpha^\ast=\bigcup_{\xi<\kappa}C_\xi$.

At the end, there is a unique condition $q=H_{\alpha^\ast} (\langle (p_\zeta, C_{\zeta}):\,\zeta<\kappa\rangle)$. This condition is no longer an element of $N$, 
but it satisfies that for every $\zeta$, we have $p_\zeta\le_{C_\zeta, \zeta}q$.

\begin{claim}\label{cl:axiomA--proper} For every $\zeta<\kappa$ we have $p_\zeta\le q$.
\end{claim}
 
 \begin{proof} By induction on $\beta\in N\cap \alpha^\ast$ we show that for every $\zeta<\kappa$ we have $p_\zeta\rest \beta\le q\rest \beta$.
 The conclusion then follows since for every $\zeta$ we have $\supt(p_\zeta)\subseteq N$.
 Also note that our assumptions imply that $N\cap \alpha^\ast$ is an ordinal.
 
 The stage $\beta=0$ is trivial. Coming at the successor stage, let $\zeta<\kappa$ be arbitrary. We need to show that
\[
q\rest\beta\forces_{{\mathbb P}_\beta}``p_{\zeta}(\beta)\le q(\beta)".
\]
Let $\zeta^\ast$ be such that $\beta\in C_{\zeta^\ast}$. Without loss of generality, we have $\zeta\le\zeta^\ast$.
Then by part (5) in the definition of the fusion sequence and the inductive assumption we have that $q\rest\beta$ forces that $q(\beta)$ is the fusion limit of
$\langle p_\xi(\beta):\,\xi \in [\zeta^\ast, \kappa)\rangle$. In particular we have the desired inequality, since 
\[
p_\zeta\rest\beta \le p_{\zeta^\ast}\rest \beta \le q\rest \beta
\]
and $p_{\zeta^\ast}\rest\beta\forces_{{\mathbb P}_\beta} ``p_{\zeta^\ast}(\beta)\ge p_{\zeta}(\beta)"$. This finishes the successor stage.

The case of $\beta$ limit follows from the induction hypothesis and the definition of the order at limit stages.
$\eop_{\ref{cl:axiomA--proper}}$
\end{proof}

It remains to check that $q$ is ${\mathbb P_{\alpha^\ast}}$-generic over $N$. 

Let $\DD$ be a dense open subset of ${\mathbb P_{\alpha^\ast}}$ with $\DD\in N$. We claim that $\DD\cap N$ is predense above $q$. 
Let $\zeta$ be such that $\DD=\DD_\zeta$ and suppose that $r\ge q$ is given.
We need to find a condition in $\DD\cap N$ comparable with $r$. We have that $p_{\zeta+1}$
is such a condition, because $p_{\zeta+1}\le q$ and $p_{\zeta+1}\in  \DD_\zeta\cap N$. 
$\eop_{\ref{cor:preser1}}$
\end{proof}
$\eop_{\ref{th:iteration}}$
\end{proof}

\section{Iterations with Amalgamation}\label{sec:whichones}
To show that our general iteration theorem is meaningful, we give an example of an iteration with amalgamation. It is the iteration of the Perfect with Respect to a Filter Forcing. 
For such a thing to make sense we need to make an additional assumption, which is that the filter $\FF$ we intend to use remains a $(<\kappa)$-complete filter on $\kappa$ after forcing with $(<\kappa)$-closed $B(\kappa)$-proper forcing. We note that this assumption is necessary, because for all we know,
$\FF$ might obtain new elements upon such forcing.\footnote{This particular assumption is relevant for many of the forcing arguments which discuss the iteration of a fixed forcing notion, often by tacitly and wrongly assuming that the forcing order itself does not get new elements.} Hence in this section we assume that 
\begin{itemize}
\item $\FF$ is a $(<\kappa)$-complete filter on $\kappa$ which does not get new elements upon forcing by a $(<\kappa)$-closed $B(\kappa)$-proper forcing.
\end{itemize}
It follows that any sequence of $<\kappa$ many sets elements of $\FF$ which is in the extension is already in the ground model, and hence its intersection is in $\FF$ because $\FF$ is $(<\kappa)$-complete in the ground model. So $\FF$ has no new elements, no new $(<\kappa)$-sequences of elements and in particular is still a $(<\kappa)$-complete filter in the extension. In particular, the forcing $\name{\mathbb P}(\FF)$ is forced to be well-defined in the extension and to satisfy all the properties of ${\mathbb P}(\FF)$ studied in this paper.

\begin{example}\label{ex:remainsOK} An example of a filter satisfying the assumptions above is the filter of the co-bounded subsets of $\kappa$. Namely,  
$\FF$ does not get any new elements upon forcing by a $(<\kappa)$-closed forcing, since no new bounded - and hence co-bounded- subsets of $\kappa$ are added by such a forcing. 
\end{example}

\begin{theorem}\label{th:amlgamation} The iteration with supports of size $\le\kappa$ of ${\mathbb P} (\FF)$-forcing has amalgamation.
\end{theorem}

\begin{proof} Let $\bar{(\mathbb P}=\langle {\mathbb P}_\alpha, {\mathbb P} (\name{\FF}_\alpha)):\,\alpha< \alpha^*\rangle$ be an iteration of 
${\mathbb P} (\FF)$-forcing with supports of size $\le\kappa$ and $ {\mathbb P}_{\alpha^\ast}$ its limit. 
We shall use the notation $\name{Q}_\alpha$ for ${\mathbb P} (\name{\FF}_\alpha)$. Let $p\in {\mathbb P}_{\alpha^\ast}$.
By extending $p$ if necessary, we may assume that $C\subseteq \supt(p)$, where $C$ is a set as in Definition \ref{def:amalgamation}.

Let $N\prec \HH(\chi)$ satisfy $|N|=\kappa$ and $N^{<\kappa}\subseteq N$, while $p, C, \DD, \bar{{\mathbb P}} \in N$. In particular, $\supt(p)\subseteq N$.

By the argument explained in Observation \ref{obs:masteq}(2), it suffices to assume that $\DD$ is the dense set of conditions deciding the value of $\name{\tau}$ for some ordinal (hence without loss of generality $\name{\tau}\in N$ by elementarity) and to show that we can find a $q\ge p$, $q\ge_{C, \zeta} p$ and a set $X\in {\mathbf V}$ of size $\le\kappa$ with 
$q$ forcing $\name{\tau}$ to be in $X$.

We consider the set of the representatives of the equivalence by forcing classes of all functions $f$ defined on $\supt(p)$ 
such that: for every $\beta\in \supt(p)\cap N$, it is forced by $p\rest\beta$
that there is $\name{s}$ in the $\zeta$-front of $p(\beta)$ and an extension $f(\beta)=\name{r}_s$ of $\name{s}_{p(\beta)}$ satisfying that for some condition
$r$ we have
\begin{itemize}
\item $r\ge p$,
\item $r\in \DD\cap N$ and
\item $(\forall \beta\in C)\,r\rest\beta\forces_{\mathbb P_\beta} {``}r(\beta)\ge f(\beta)"$.
\end{itemize}
For each such $f$ we choose a condition 
$r_f\in N$ witnessing the above conditions. We let $X$ be the set of all values forced to $\name{\tau}$ by any such $r_f$,
which gives us a set of size $\le\kappa$, since we have chosen $r_f\in N$ and $|N|=\kappa$. (The details of this argument are explained in the proof of Lemma \ref{Miller-density}).

We let $q$ be the condition defined as follows: $\supt(q)=\supt(p)$ and for every $\zeta\in \supt(p)$ the value of 
$q(\beta)$ is forced by $q\rest \beta$ to be 
 the $\zeta+1$- amalgam of $\{r_f(\beta):\,f\in \FF\}$. We extend $q$ to a condition whose support is $N\cap \alpha^\ast$ so that the value
 of $q(\beta)$ for $\beta\in N\setminus \supt(p)$ is the $1$-amalgam of $\{r_f(\beta):\,f\in \FF\}$.
 
We claim that  $q\forces {``}\name{\tau}\in X"$. Indeed, suppose $r\ge q$ and $r$ forces a value to $\name{\tau}$. We have defined $q$ so that $r$ must be compatible with $r_f$ for some $f\in \FF$, since $r_f\in N$ implies $\supt(r_f)\subseteq \supt(q)$. Therefore the value forced to $\name{\tau}$ by $r$ is already in $X$.
$\eop_{\ref{th:amlgamation}, \ref{th:iterationPrikry} }$
\end{proof}

\begin{corollary}\label{th:iterationPrikry} The iteration with supports of size $\le\kappa$ of ${\mathbb P} (\FF)$-forcing for $\FF$ the filter of co-bounded subsetes of $\kappa$ is strongly 
$(<\kappa)$-closed and $B$-$\kappa$-proper.
\end{corollary}

While we have followed the tradition concentrating on the forcing iterations where the filter remains fixed throughout the iteration, this has actually only made our life harder by having to consider only very special filters such as the one consisting of co-bounded sets. Allowing $\name{\FF_\alpha}$ to just be a name of a $(<\kappa)$-complete filter on $\kappa$, as we anyway did in the proof of Theorem \ref{th:amlgamation}, we obtain by exactly the same proof the following Corollary \ref{cor:mixed-filters}.

\begin{corollary}\label{cor:mixed-filters} Suppose that $\bar{\mathbb P}=\langle {\mathbb P}_\alpha, {\mathbb P} (\name{\FF}_\alpha)):\,\alpha< \alpha^*\rangle$ is an iteration with supports of size $\le\kappa$ of forcing of type
${\mathbb P} (\FF)$, where the filter at the stage $\alpha$ is named by the forcing ${\mathbb P}_\alpha$ to be a $(<\kappa)$-complete filter on $\kappa$. Then $\bar{\mathbb P}$  has amalgamation
and ${\mathbb P}_{\alpha^\ast}$ is is strongly 
$(<\kappa)$-closed and $B$-$\kappa$-proper.
\end{corollary}

We note that in the recent work \cite{Kanovei-style} we obtain a generalisation of Theorem \ref{th:amlgamation} and Corollaries \ref{th:iterationPrikry} and 
\ref{cor:mixed-filters} to iterations along any well-founded order. 

\section{Conclusions and Further Directions}\label{sec:conclusions}

\subsection{Conclusions}\label{subsec:conclusions}
We have introduced an uncountable variant of the Perfect Set Forcing over a Filter (Generalised Perfect Set Forcing) and have explored its behaviour under products and iterations.
Our main results are Theorem \ref{th:amlgamation} and Corollary \ref{th:iterationPrikry} which give an iteration method for iterating 
Perfect Set Forcing over a Filter and a template for iterating its abstract variants.

The method of the proof of Theorem \ref{th:amlgamation} is not that of a classical iteration theorem, since the property of having amalgamation is required
to hold of the whole iteration rather than of the individual iterands. Although we have come up with this kind of iteration in a different way, a posteriori it resembles the iteration method that Kanovei introduced in \cite{zbMATH01335192} for being able to iterate Perfect Set Forcing along any partial order. This method has in the meantime gotten known as {\em geometric iteration} and has been used by many authors. It is in fact the proof that arbitrary iterations of Perfect Set Forcing can be considered as a sort of {\em long products}. This is no longer true for the Perfect with Respect to a Filter Forcing, either by generalising it to other arboreal forcings on $\omega$ or by generalising it to uncountable cardinals. An iteration theorem for the Perfect Set Forcing over the filter of co-bounded subsets of $\kappa$ along arbitrary well-founded orders, while preserving $\kappa^+$, is obtained in \cite{Kanovei-style} .

\subsection{Further Directions}\label{subsec:further-directions}
There is likely nothing optimal about our method, although it is known that some
extra condition has to be put in order to obtain an iteration theorem about $\aleph_1$-proper countably closed forcing, see the example due to Inamdar in
\cite{paper34}. It is not very clear where the limit between the technical difficulty of preservation in the limits of countable cofinality and the actual impossibility of having an iteration theorem lies. An upcoming survey paper by Inamdar and Rinot should deal with this problematics. 

As for our iteration theorem, the next direction to explore would be if this method can be used in situations where not all individual forcings are of the same kind. In our opinion it should be doable, with some extra effort. However, one should take into account that iterating Sacks forcing can be done even along any kind of partial order, as proven by Kanovei in \cite{zbMATH01335192}, while the general Axiom A is hard to iterate. 
If a more general method for iterations should be found, only then we could talk about a possible forcing axiom, which would necessarily be expressed in a different and less general way than the classical forcing axioms which apply to all members of a given class of forcing, no matter in what order or combination we take them.
A possible philosophical advantage would be that such a limited forcing axiom would not require the use of large cardinal axioms, since the project of generalising PFA to $\aleph_2$, whatever limited form it might be taking in the present state of the research, would seem to be obtainable basically from two supercompact cardinals. This statement is modelled on what we currently know of the large cardinal strength of PFA. It is an open conjecture if the full power of supercompactness is needed for the consistency of PFA, but it is known that PFA has large cardinal
strength. To be exact, the consistency strength of PFA is at least a proper class of strong cardinals and a proper class of Woodin cardinals, by a result of
Ronald Jensen, Ernest Schimmerling, Ralf Schindler and John Steel in \cite{Jensen_Schimmerling_Schindler_Steel_2009}. 

It is to be noted that the original proof of Baumgartner that it is consistent that every two $\aleph_1$-dense sets of reals are isomorphic does not use PFA, but a composition of $\sigma$-closed and ccc forcing, both of which are Axiom A. Hence it is conceivable that an appropriate use of 
Theorem \ref{th:iteration} could generalise Bamgartner's result to $\aleph_2$-dense sets of reals. It has often been said that this question was Neeman's motivation in discovering forcing with two types of side conditions and might have already been solved by his method, although there is no published proof at this time. Perhaps our method might give an easier proof.

Going in a different direction, it was proven in Baumgartner and Laver \cite{zbMATH03664937} that by starting with a weakly compact $\kappa$ and adding 
$\kappa$ Sacks reals iteratively, the final model resembles the Mitchell \cite{Mitchellmodel} model: $\kappa$ is collapsed to $\aleph_2$ and there are no $\aleph_2$-Aronszajn trees. Jack Silver had proven earlier, in \cite{zbMATH03356764}, that a large class of forcings have the main property used by Mitchell to obtain this result.\footnote{Perhaps due to the fact that Silver's paper is not readily available, the fact seems to have remained unnoticed. See Cummings 
\cite[Lem. 1]{zbMATH06836951} for a more readily available rendition of that proof.} It is reasonable to think that a similar property should hold for the iteration of the forcing ${\mathbb P}(\FF)$ as done in this paper, which would then give a new model with no $\aleph_3$-Aronszajn trees. Originally, a model with neither $\aleph_2$ or $\aleph_3$-Aronszajn trees was obtained by Uri Abraham in
\cite{zbMATH03815624}, who starting from a supercompact cardinal and a weakly compact cardinal above, used Micthell's forcing on both to achieve the desired results. Many more results on the tree property have been obtained in the meantime, often using variants of Mitchell's forcing, see the surveys \cite{CummingsForemantree} by Cummings and Foreman and a more recent one by Dima Sinapova \cite{Sinapovaomeg2}. It would be interesting to see if similar results could be obtained using the Perfect with Respect to a Filter Forcing. Perhaps by replacing Mitchell forcing by the generalised Perfect Set Forcing the known difficulties of achieveing the tree property at the successor of a singular cardinal could be overcome.

\section{Historical Remarks}\label{sec:history} As the reader will notice, the results of this paper
would not have been possible had we not taken a very particular care about the exactness of definitions and fine points that if missed can lead to wrong conclusions. On the other hand, much of the set-theoretic literature, including the textbooks, contains imperfections that  have led to wrong directions that have just been assumed true. Therefore we feel that it is very important that at least part of it be stated somewhere. The literature also contains a strong element of wrongly given credits, which we feel is important to correct both for moral and for mathematical reasons.

\subsection{On Perfect Set Forcing and Generalisations}
The first arboreal forcing formulated in terms used now was discovered by Gerald Sacks in \cite{Sacks}, who called it the Perfect Set Forcing. He was interested in the concept from the point of view of recursion theory, since this forcing adds a real of the minimal degree. The main ideas of 
Perfect Set Forcing are present in Albert Muchnik's solution of the Post problem \cite{zbMATH03117565}. Richard M. Friedberg published a paper \cite{zbMATH03131930} on the same solution of the Post problem one year after Muchnik, so we do not believe that he should be credited with the solution, which is often attributed to both authors independently.

The fact that a particular arboreal forcing was preserved by countable support iteration was discovered by Laver in his paper on the consistency of the Borel Conjecture \cite{Laver_Borel}, where he iterated what we now call Laver forcing. Baumgartner and Laver \cite{zbMATH03664937} proved the iteration theorem showing that iterations with countable support of Perfect Set Forcing preserve $\aleph_1$.
In his lectures in Cambridge 1978 and the paper \cite{Ba} recording them, Baumgartner proved that iterating Axiom A forcing with countable support iterations preserves $\aleph_1$. 

An important precedent to our paper is the paper by Akihiro Kanamori  \cite{zbMATH03708389} in which he develops a generalisation of Perfect Set Forcing to uncountable cardinals $\lambda$. He carefully discusses the problems of the iteration and how to resolve them using a 
$\diamondsuit$ sequence in the ground model. However, his closure and fusion arguments for a single forcing are incorrect, starting from Lemma 1.2. This might be due to the lack of correctly written fusion arguments for the Perfect Set Forcing at the time. Namely,
Sacks' paper \cite{Sacks} developed forcing with the perfect sets of reals. Both this forcing and its translation in terms of perfect trees appeared in Thomas Jech's textbook's first edition \cite{Jec78}, but it left the fusion argument in terms of trees as an exercise (Exercise 26.4). The first written representation for the fusion argument for forcing with Perfect Trees appears in Baumgartner's article \cite{Ba}. Unpublished version of the argument using a topological argument was known to Mathias around 1971. The topological argument for fusion was used in the Soviet school from the very beginning of forcing. Versions of Kanamori's forcing appeared in various other articles under the name of Sacks$(\kappa)$, starting from \cite{zbMATH05531005}, but they also get the closure argument wrong. The first written evidence of somebody finding these errors is by Sy-David Friedman and Victoria Gitman \cite[\S 2]{FriedmanGitman:ModelOfACNotDCInaccessible}, who correct them and develop the $\kappa$-Perfect Set forcing for $\kappa$ inaccessible, forcing with subtrees of ${}^{<\kappa} 2$. Their forcing does not fit the framework of Property B studied in this paper, which needed the version involving splittings into a filter.

It is not mentioned in their paper, but it is reasonable to suppose, that the ``proper over semi-diamonds" notion of Ros{\l }anowski and Shelah was inspired by Kanamori's proof, which hence does work if the forcing is $(<\lambda)$-closed. Kanamori gives credit to Baumgartner who gave him the advice to 
use $\diamondsuit$ and had used it himself in the paper \cite{zbMATH03529856}. 

Perfect Set Forcing over a Filter appears in a paper by Brown and Groszek \cite[\S1]{zbMATH05121369}. The notions are cited as due to Brown in her thesis from 2006 and in an unpublished paper \cite{Brown}. Brown and Groszek refer to the forcing as the generalised Miller forcing and study it with respect of adding a minimal degree.
We note that their Definition 1 misses the requirement that the elements of the forcing are trees of height $\kappa$, which is necessary for the generic function to be defined. Cardinal preservation properties are not studied in \cite[\S1]{zbMATH05121369}. In particular, proofs of $(<\kappa)$-closure and the preservation of 
$\kappa^+$ are not given in \cite{zbMATH05121369}, although it is stated that they will be given in \cite{Brown}, which is still unpublished and unavailable.

As for Grigorieff forcing, we do not think that Proposition 5.18 in \cite{Jec86} which claims that $\sigma$-product of Grigorieff forcing preserves $\aleph_1$ is correct. That proposition, given without a proof, refers to the set
rather than the tree version of Grigorieff forcing. The two are closely related and in fact the same proof that \cite{Jec86} gives for Proposition 5.17 showing that $\sigma$-product of infinitely many Mathias forcing collapses $\aleph_1$ applies to the set version of Grigorieff forcing. The original paper of Serge Grigorieff \cite{zbMATH03513777} also includes mistakes, including an incorrect characterisation of situations when this forcing adds a minimal real.

\subsection{On Axiom A and Proper Forcing}
Axiom A forcing and proper forcing have similar properties and interwined histories. Baumgartner introduced Axiom A in a series of invited lectures in Cambridge in the summer of 1978. They were published as \cite{Ba}. The article corresponds exactly to what was said in the lectures, according to a statement from Magidor \footnote{e-mail, 27 February 2025}, who was present and a co-organiser of the conference. There is no publicly available information on the reason for the five year publication gap, which  is unusually long. 

Shelah introduced proper forcing in various lectures in the period 1978-1980 and published the work in \cite{Sh100}. The 1998 Shelah's book \cite[pg. 399]{Sh_P} states that Baumgartner's work on Axiom A and Shelah's work on proper forcing were discovered independently of each other and at about the same time. The ``About this book" on Springer's web site \url{https://link.springer.com/book/10.1007/978-3-662-21543-2} of the first edition of this book \cite{zbMATH03779315}, states that the author lectured on the subject at the University of Berkeley in 1978 and at the Hebrew University of Jerusalem 1979-1980. Keith J. Devlin, who contributed an article on proper forcing  \cite{zbMATH03829897} to \cite{zbMATH03805451}, states on pg.66 that he was not present in the Cambridge meeting and started writing his article in 1980, but that the first ``worldwide exposure" of proper forcing happened in that meeting, since Shelah ``spoke about such matters". Devlin states that his article is based upon the ``original, rather sketchy notes by Shelah" on 1978 Berkeley lectures, with considerable improvements by Baumgartner in 1980 and Todor{\v c}evi{\'c}'s correction of a mistake in the penultimate version. 

It seems fair to conclude that while Axiom A was presented in its finished form in the Cambridge meeting in the summer of 1978, ``the whole concept (of properness) was 
still in its infancy" at the time, to cite Devlin. In fact, the various bibliographic information suggests that the 1978 meeting where Baumgartner exposed his finished work on Axiom A was the moment when Shelah conceived his idea of proper forcing.  

Let us explain the connection between the two notions. Shelah \cite[III Lemma 2.6]{zbMATH03779315} claims that {\em Mastering $\aleph_0$} with
$\le_\zeta$ in its definition replaced by $\le$ is equivalent to properness. It is true that properness implies {\em Mastering $\aleph_0$}, with a proof similar to that of Lemma \ref{implication-properness} here, however the proof of the other direction of \cite[III Lemma 2.6]{zbMATH03779315} is wrong,  
and the statement is wrong as well: as pointed out by Goldstern\footnote{e-mail 09 September 2025, including a proof we agree with}, every forcing satisfies the property that is claimed to be equivalent to properness. In the second edition of \cite{zbMATH03779315}, which is \cite{Sh_P},  without giving comments on the mistake in the previous edition, Shelah gives the correct statement with the implication going in only one direction, as \cite[III, Lemma 1.16]{Sh_P}. Given what we have just said, this statement is trivial. Goldstern \cite[Fact 3.10]{Goldsterniteration} gives a translation of 
Shelah's \cite[III Lemma 2.6]{zbMATH03779315} into a related correct statement which characterises a $\mathbb P$-generic condition over $N$ as one that for every 
name of an ordinal which is in $N$ forces the evaluation to be in $N[\name{G}]$- the difference being that we are not dealing with one ordinal at a time but all
countably many ordinals at once.  

Shelah proved that properness is preserved under countable support iterations \cite[Theorem 3.2]{Sh_P}, with a proof analogous to Baumgartner's proof \cite[Theorem 7.1]{Ba} of the fact that Axiom A forcing preserves $\aleph_1$ and using the main idea of Baumgartner's proof of building fusion sequences
in order to pass the limit stages of countable cofinality. This idea comes from the Baumgartner-Laver work \cite{zbMATH03664937} on iterated Perfect Set Forcing, published in 1979. However, the definition of fusion needed for the iteration of Perfect Set Forcing in  \cite{zbMATH03664937} is incorrect (Lemma 1.1. and just above), and it is only corrected in \cite{Ba} (without any comments on the different definition of the fusion or mention of \cite{zbMATH03664937} being incorrect). Compared to Baumgartner's proof in \cite{Ba}, the proof in Shelah's version is reformulated in terms of elementary submodels. The advantage of properness is that since we are
not requiring the strong form witnessed by $q\ge_\zeta p$ for a pregiven $\zeta$, but only $q\ge p$, we do not lose strength in the limit by having to increase
$\zeta$. As a result we obtain an actual iteration theorem where the property of being proper is preserved in the limit, rather than just a cardinal preservation theorem. 

Tadatoshi Miyamoto showed in his 1988 Ph.D. thesis, the relevant part of which is published in his paper \cite{zbMATH04139722}, that by modifying Axiom A to a somewhat stronger condition, we can iterate it with countable supports through an
iteration of length of $\omega_2$ and preserve that stronger property. Koszmider in \cite[\S4]{zbMATH00238211} used a different method to iterate Axiom A to any length, but unfortunately there is a mistake in the proof of Proposition 20, definition of $G_A$.  Kanovei proved in \cite{zbMATH01335192}
that Perfect Set Forcing can be iterated along any partial order. Tetsuya Ishiu claims in \cite[Th 4.3]{zbMATH02211269} to have proved that Axiom A is (forcing) equivalent to $(<\omega_1)$-proper forcing, but his proof breaks down since his notion of uniform A axiom is simply being forcing equivalent to either a ccc or a countably closed forcing, so is not equivalent to Axiom A. \footnote{Ishiu in his e-mail 29 September 2025 says that he knew that Theorem 4.3 was wrong and proposes a patch up for his notion of uniform A forcing, which we think is still wrong.} The fact that Axiom A forcing is 
$(<\omega_1)$-proper is true and is due to Miyamoto \cite{zbMATH04139722}. The incorrect theorem of Ishiu is cited and its supposed impact explained in the paper \cite[pg. 1179]{zbMATH06476477} by David Aspéro, Sy-David Friedman, Miguel Angel Mota and Marcin Sabok.

As explained in the above and shown already by Baumgartner \cite[Th 2.4]{Ba_AppPFA}, Axiom A is a particular case of properness. Hence, in the years after the invention of properness, Baumgartner's method was overtaken by the powerful proper forcing machinery. There is an axiom associated to proper forcing, the Proper Forcing Axiom (PFA), whose consistency modulo a supercompact cardinal was proved by Baumgartner. It is common to reference Baumgartner's article \cite{Ba_AppPFA} for this proof, but actually the proof is not there. There it is stated that the paper is only about applications of PFA and that the proof is available in 
Devlin \cite{zbMATH03829897} and Shelah  \cite{zbMATH03779315}. The proof is indeed available in Devlin \cite{zbMATH03829897} with no credit given. 
Matthew Foreman, Magidor and Shelah in \cite{FMS} credit the proof to Baumgartner. The consistency proof of PFA does not appear in \cite{Sh_P} but the consistency proofs of various other similar axioms which do not need to use the Laver diamond do. It is a common understanding that the proof is indeed due to Baumgartner, in particular nobody else has claimed the proof. 

It seems fair to conclude that the use of Laver's diamond in the consistency proof of PFA and other forcing axioms is Baumgartner's idea, which Shelah did not have when considering forcing axioms in his book \cite{zbMATH03779315}.

The machinery of proper forcing depends on using countable elementary submodels of a large enough model of a sufficient amount of ZFC.
A nice property of such countable models $M$ is that their intersection with $\omega_1$ is a countable ordinal $\delta$. This property is used in various 
proofs, including the preservation under countable support iteration. So when looking for a generalisation of proper forcing to cardinals 
of the form $\lambda^+ >\aleph_1$, it is natural to consider a cardinal $\lambda$ with $\lambda=\lambda^{<\lambda}$ so as to have models of size $\lambda$ whose intersection with $\lambda^+$ is an ordinal. Hence, one can try iterating with supports of size $\lambda$ forcing notions that are $(<\lambda)$-closed,
to preserve $\lambda=\lambda^{<\lambda}$, and have some analogue of properness. There are ZFC results that show that we cannot hope to have a full analogue of properness at such cardinals $\lambda$. For example
Shelah showed
(\cite[Th 2.1]{shelah_diamond_1980}) that  $\diamondsuit(\lambda^+)$ follows from GCH, while one can force using proper forcing a model of 
CH in which $\diamondsuit$ in which fails. The first model of this was obtained by Shelah in \cite{shelah_whitehead_1977}, but any model where CH holds and there are no Suslin trees also gives what is needed. Such a model constructed by Jensen is announced in \cite{Jensencover} and a proof is given in \cite{zbMATH03453594}. A different proof by Shelah is in [ChV]\cite{Sh_P}.
 
As mentioned before, Ros{\l }anowski and Shelah in \cite{zbMATH01751288} developed a theory of forcing which is $(<\lambda)$-closed and ``proper over semi-diamonds", for $\lambda=\lambda^{<\lambda}$. They proved (Th. 2.7 \cite{RoSh655}) that
such forcing can be iterated with supports of size $\le\lambda$ and wrote a series of papers continuing these ideas. This seems directly inspired by Baumgartner's use of $\diamondsuit$ for the (incorrect) proof of cardinal preservation of the forcing $R(\kappa,\lambda)$ in his article \cite{zbMATH03529856}.

\subsection{On Generalised Martin's Axiom}
Baumgartner in \cite[\S4]{Ba} and Shelah in \cite{Sh80} both introduced (different) generalisations of Martin's Axiom MA that hold at $\aleph_2$
and other successors of regular cardinals. Baumgartner states that the first version of the generalised MA was an unpublished result of Laver and that Curtis Herink \cite{Herink} has an improved version of the Baumgartner's Axiom. Franklin Tall indicates in his 1994 paper \cite{zbMATH00600039} on applications
of generalised MA that versions of it existed already in 1977 in the form that they could be applied, but we have not found any written evidence of that.
Versions of the axioms have appeared in various works since that time, notably the stationary $\kappa^+$-cc
forcing by Shelah, and all have the same flavour as the generalisations cited here: a strong version of $\kappa^+$-cc, $(<\kappa)$-closure and an additional condition to help pass the stages of countable cofinality. It seems fair to state that these generalisations were around in one way or another by 1977 and that the idea of having them was Laver's. In the recent paper \cite{Mirna-Novosibirsk} we have shown that neither the well-metness condition from \cite{Sh80} nor the countable parallel closure from \cite[S2]{5authorsforcing} are necessary (nor sufficient), and that instead we have to consider a definable version of closure, which is what we call the strong closure here.

Baumgartner \cite[\S4]{Ba} says that one would hope that there should be an iteration theorem for $\aleph_2$-cc countably closed forcing over a model of CH, preserving cardinals in the iteration and with no additional requirements on the iterands\footnote{There seems to have been a habit of the time to publish statements supposed to say exactly the opposite of what one thinks is true. We believe that this was done for the reasons of expressing a revolt at the various unfounded mathematical claims floating in the community. So, it is possible that Baumgartner already knew that such a full axiom is not possible, but preferred to let others show their lack of knowledge on the subject by struggling to find the proof themselves. Unfortunately, we cannot ask Baumgartner in person as he is no longer alive. Hence we take the liberty to take him for his word and call his statement a conjecture.}. We take the liberty to call this statement Baumgartner's conjecture.
The supposed evidence to the contrary appears in the form of various claims that it is well known that one cannot iterate $(<\kappa)$-closed and $\kappa$-proper forcing with supports of size $\kappa$, whatever local improvements on these notions one might try. The support for this claim is given by citing Shelah's \cite[Appendix 3.4A]{Sh_P}, which is supposedly giving a proof to a statement from \cite[XIV, 3.4]{zbMATH03779315} given without a proof. However, Shelah's argument is certainly unclear to us and the claim accompanied by the comment below it cannot be correct since it would contradict Baumgartner's version of generalised MA given in \cite[\S4]{Ba}. The claim without the comment is indeed proven in a blog post by Assaf Rinot \cite{Assafblog} and the proof is highly non-trivial. However, this still does
not answer Baumgartner's conjecture. The conjecture was said to be answered in the negative by Kunen by an  Exercise in his book  \cite[VIII, E6]{Kunen}, but since there is no proof given and the exercise does not seem obvious to us, we cannot state that this is indeed the solution. In addition, the revised version of this book \cite{Kunen2}, which contains solutions to many exercises from the original edition, no longer has this exercise (or a solution). The 34 year-old conjecture of Baumgartner was finally refuted by Inamdar in 2017, see the argument with the citation in the paper \cite{paper34} by Chris Lambie-Hanson and Rinot, or Inamdar's Ph.D. thesis \cite[Ch 10]{Tanmaythesis}. 

\subsection{On Iterated Forcing and the Semantics of Forcing}
Some, but not all of the difficulties of finding an iteration theory for forcing with supports of size $\aleph_1$ are related to the way that we define iterated forcing. The usual definitions, described in detail in either Jech's textbook \cite{Jech-ST}, Kunen's textbook (\cite[VIII]{Kunen}), Baumgartner's article \cite{Ba}, Martin Goldstern's article \cite{Goldsterniteration} and Shelah's book \cite[pg. 57 for FS and pg. 107 for CS]{Sh_P} are not all the same. 
The inventors of iterated forcing are Robert Solovay and Stanley Tennenbaum
in \cite{SolTan} and their definition uses iterated Boolean algebras and embeddings. 
In the various references mentioned above, an iteration with countable supports
$\langle {\mathbb P}_\alpha, \name{Q}_\alpha:\,\alpha\le\alpha\ast, \beta<\alpha^\ast\rangle$ 
where each ${\mathbb P}_\alpha$ is a forcing notion and $\forces_{{\mathbb P}_\alpha}``\name{Q}_\alpha \mbox{ is a forcing notion}"$ is defined using partial orders
(although earlier editions of \cite{Jech-ST} used a Boolean algebra approach to forcing). Neither the notation nor the definitions are the same. One can for example see the approach of Kunen and Baumgartner on the one hand, and Azriel Lévy and Shelah on the other. Shelah's book 
\cite[III Def. 3.1]{Sh_P}, which is a changed and improved version of \cite{zbMATH03779315}) uses the following as a definition: $p\in {\mathbb P}_\alpha$ is a function with domain a countable subset of $\alpha$ such that for all $\beta\in \dom(f)$ we have
that $p(\beta)$ is a canonical ${\mathbb P}_\beta$-name and $\forces_{{\mathbb P}_\beta}``p(\beta)\in {\name Q}_\beta"$. Then 
the order is defined as $p\le_{{\mathbb P}_\alpha} q$ iff for all $\beta\in \dom(p)$ we have $q\rest\beta \forces_{{\mathbb P}_\alpha}``
q(\beta)\ge_{{\name Q}_{\beta}} p(\beta)"$. Finally, it says, for all $\beta\notin\dom(p)$ we let $p(\beta)=\name{\emptyset}_{\name{Q}_j}$.

This definition lacks precision. It recognises that ``instead of $p(\beta)$ is a canonical ${\mathbb P}$-name we can use other
variants" (the end of the definition), but it leaves out the details of checking that the various properties of the iteration hold, Fact 3.1.A.

It seems fair to say that these different definitions of iterated forcing came from the desire of various competing schools to have their own version of the notion. This is certainly known to have been the case with the notion of forcing, which has had a very turbulent history involving both partial orders and Boolean algebras. See \cite{Kanamori} for a carefully researched and carefully phrased account, with a discreet personal take. 

Our own personal take is that forcing with Boolean algebras was started by Scott and then developed in 1967 in joint work by Dana Scott and Solovay\footnote{John Truss (10 September, 2025) believes that the development of this approach to forcing likely happened during the summer school at UCLA in 1967. This school is mentioned also in John Bell's book and William Mitchell cites the notes from this school as a reference in his paper \cite{Mitchellmodel}, where he uses Boolean valued forcing. Jack Silver's paper \cite{zbMATH03402625} says that the Boolean valued models approach to forcing will appear in a joint paper by Scott and Solovay in \cite{zbMATH03458513}, but there is no such paper there.}, but that Solovay then went on to work on this in an industrious, yet often unpublished, way on his own. Scott preferred to finish writing a careful presentation through a manuscript which they started and have never finished. Scott lectured on the subject in Oxford in 1967 \footnote{statement by the then student George Wilmers, given 10 September 2025.}. The work finally only appeared in John Bell's book \cite{zbMATH05910780}, whose introduction gives historical remarks. A recent book on forcing with Boolean algebras, which avoids the historical ambiguities by a radical solution of not giving any credits is Matteo Viale's  \cite{zbMATH07944208}. From our point of view, the most important development in these early accounts is that Scott's approach with Boolean algebras answers a question that none of the others mentioned here have even thought of asking: the correct semantical definition of models used in forcing, the construction of countable transitive models (on which there are many incorrect statements in the literature) and the question of where the generic extension actually takes place. Scott's take on this later led him to ``domain theory". 

At any rate, it is known that forcing with partial orders (once the correct semantics is given) is basically the same as forcing with Boolean algebras, but it is not quite clear that this should be the case with iterated forcing, especially if large supports are considered.
However, all authors end up treating each other's work as if they were using the same definition, which ordinarily is assumed known and not repeated every 
time an iterated forcing construction appears in many papers that use it as a tool. As a consequence, some paradoxes appear. They are much less 
serious in the case of finite support iterations, where various notions are dense in each other as partial orders, but this comfortable fact disappears when we pass to iterating forcing with larger supports. Kunen \cite[VIII, Ex E]{Kunen} very carefully discusses these paradoxes and proposes a solution to some of them using full names. Goldstern \cite{Goldsterniteration} gives a very detailed and precise account of countable support iterations, not dwelling on the differences in the previous literature but just giving a correct presentation and many applications. The difference between the various possible definitions is completely unappreciated by some of the authors, such as in the citation by Shelah, above. 

The first person to state these differences was Chaz Schlindwein in his Ph.D. thesis in 1993 at Pennsylvania State University, that was then published as the paper \cite{Chaz}, where he proposed another definition of iterated forcing. Unfortunately, after ten years of using this method which gave very nice results, Schlindwein understood that he made a fundamental mistake, see \cite{zbMATH02056964}, which affects all the results he obtained. The proposed patch up to the mistake, in the same paper, is not worked out. Parts of his approach are revived by Miyamoto in \cite{Miyamoto}. This is quite a technical paper which goes beyond the scope of our work here, so we have not given it a detailed review. 

The definition of iterated forcing that we use follows Baumgartner \cite[\S4]{Ba}, with our own unpublished account of obtaining countable transitive models and finding the semantic domain of discourse for generic extensions. As we discovered in later discussions with Kanovei, this is the same that he had discovered since the early days, without ever publishing it. 

\subsection{On Forcing with Large Supports}
As explained in the introduction, forcing with uncountable supports presents additional complications in the iterated forcing arguments. 

Some authors, including Groszek and Jech  \cite{zbMATH04187795} and Shelah in his 80th birthday celebration lecture in Vienna 2025, have considered  ill-founded iterations in order to resolve problems brought by an ordinary iteration, however with no concrete bearing on the problems considered here. Hence, in spite of their appeal, we have not found a use of these methods. 

\subsection{Further Possible Connections}
There is a paper that in a sense deals with a more complex situation than the one we with deal here in limits of countable cofinality, brought to our attention by Martin Goldstern. Namely, Jakob Kellner and Goldstern in \cite[Th 2.4]{zbMATH05023769}  prove that the property of being proper ``densely-preserving", as defined in their paper,  is preserved by countable support iterations. This solved some long-standing open problems of preservation connected to the countable support iterations of proper 
$\omega^\omega$-bounding forcing. The properties they consider seem connected to that of Axiom A, as they address the situation of an increasing 
$\omega$-sequence of relations whose union is the relation under consideration. However, they deal with countable support iteration of proper forcing and their relations are on ${}^{\omega}\omega$ rather than on the forcing itself. We have not been able to find the connection between that work and what we do here.

\section*{Acknowledgements}  The author submitted an unfunded proposal to ERC in April 2022, where this work formed one of the topics. She would like to thank Boban Veli{\v c}kovi{\'c} who read the proposal at submission and made a true remark about the fact that one cannot expect to iterate countably closed,
$\aleph_2$-version of Axiom A forcing with supports of size $\aleph_1$ without collapsing $\aleph_2$. He suggested that the Ros{\l }anowski-Shelah method of semi-diamonds and side conditions might be useful to solve this problem\footnote{an e-mail message, 25 April 2022.}, but the paper does not use these ideas. The author thanks the referees of the ERC proposal for useful remarks. She sincerely thanks Andrés Villaveces for engaging with her ideas, in particular in August 2022 during her visit to Colombia, and for his patience and friendship. 

While this paper was written, the author discussed it with several mathematical friends, namely Omer Ben-Neria, Martin Goldstern, Tanmay Inamdar, Jakob Kellner, Assaf Rinot and Judith Roitman. Their comments have helped a lot and I sincerely thank them for being on the right side of the broad mathematical board where this research took place.

This research would not have been possible without the membership in the libraries of the Bibliothèque Nationale de France and Institut Henri Poincare in Paris as well as in the Association of Symbolic Logic (ASL) and Math-Net.Ru, where the latter two provided online access to their journals. The memberships in the Bibliothèque Nationale de France and the Association of Symbolic Logic (ASL) were paid, while those in the Institut Henri Poincare and Math-Net.Ru were given to me upon a successful application. Therefore I most sincerely thank both of these organisations. The paper used the AI product GROK by X-AI, which served as a research assistant for rough historical facts, finding references and editing questions. All mathematical developments are due entirely to the author unassisted by the AI.

\bibliography{../../bibliomaster}
\bibliographystyle{plainurl}

\end{document}